\documentclass{amsart}
\usepackage[english]{babel}
\usepackage{color,latexsym,amsfonts,amsmath,graphics,amsthm,amssymb}

\makeatletter
\@namedef{subjclassname@2020}{\textup{2020} Mathematics Subject Classification}
\makeatother

\title[Carath\'eodory distance-preserving maps]
{Carath\'eodory distance-preserving maps between bounded symmetric domains}
\author{Bas Lemmens}
\thanks{School of Engineering, Mathematics and Physics,
University of Kent, Canterbury CT2 7NX, UK (\texttt{B.Lemmens@kent.ac.uk})}
\author{Cormac Walsh}
\thanks{Inria and CMAP, Ecole Polytechnique, CNRS, 91128, Palaiseau, France
(\texttt{cormac.walsh@inria.fr}).
Cormac Walsh was partially funded by ANR-23-CE40-0012-03.}


\date{\today}

\newcommand\after{\circ}
\newcommand\intersection{\cap} 

\newcommand\R{\mathbb{R}}
\newcommand\C{\mathbb{C}}

\newcommand\N{\mathbb{N}}

\newcommand\gromprod[3]{( #1, #2 )_{#3}}
\newcommand\triprod[3]{\{ #1, #2, #3 \}}
\newcommand\rank{\operatorname{rank}}
\newcommand\smallsquare{\mathrel{\scriptstyle \square}}
\newcommand\indexfn{J}
\newcommand\chainlinked{\sim}
\newcommand\id{\operatorname{Id}}

\newcommand\real{\operatorname{\mathrm{Re}}}
\newtheorem{theorem}{Theorem}[section]
\newtheorem{Definition}[theorem]{Definition}

\newtheorem{lemma}[theorem]{Lemma}

\newtheorem{corollary}[theorem]{Corollary}
\newtheorem{proposition}[theorem]{Proposition}

\newenvironment{definition}{\begin{Definition}\rm}{\end{Definition}}

\begin{document}

\begin{abstract}
We study the rigidity of maps between bounded symmetric domains
that preserve the Carath\'eodory/Kobayashi distance. 
We show that such maps are only possible when the rank of the co-domain
is at least as great as that of the domain.
When the ranks are equal, and the domain is irreducible, we prove that
the map is either holomorphic or antiholomorphic.
In the holomorphic case, we show that the map is in fact a triple homomorphism,
under the additional assumption that the origin is mapped to the origin.
We exploit the large-scale geometry of the Carath\'eodory distance and use the horocompactification and Gromov product to obtain these results without requiring any smoothness assumptions on the maps.
\end{abstract}

\maketitle

\newcommand\opencone{C}
\newcommand\closedcone{\closure C}
\newcommand\statespace{K}
\newcommand\filter{F}
\newcommand\filters{\mathcal{F}}

\section{Introduction}

An interesting problem in the theory of several complex variables
is to find conditions under which every map $\phi\colon D\to D'$
between two complex manifolds $D\subseteq \C^m$ and $D'\subseteq \C^n$
preserving the Carath\'eodory or Kobayashi distance is either holomorphic
or anti-holomorphic.
This problem has been considered in various settings,
often under the assumption that $\phi$ is $C^1$-smooth or $D=D'$;
see~\cite{CZ, Ed2, GS1, jarnicki_pflug_book, SV, Zw}.
Arguably, the most general known case to date is due to Antonakoudis~\cite{An},
who showed that it holds for maps between complete disc rigid domains
in $\C^n$; see also \cite{Ed1}.
It is generally believed \cite[Conjecture 5.2]{GS2} that as long as the domains
$D$ and $D'$ are not biholomorphic to a Cartesian product of domains,
the distance-preserving map $\phi\colon D\to D'$  is either holomorphic
or anti-holomorphic. 

Here, we study the case of
\emph{Hermitian symmetric spaces of non-compact type},
or equivalently \emph{bounded symmetric domains}.
In these spaces, the Carath\'eodory and Kobayashi distances coincide.
The work by Antonakoudis \cite[Theorem 1.3]{An} implies that
Carath\'eodory distance-preserving maps between rank one
bounded symmetric domains are either holomorphic or anti-holomorphic.
Distance-preserving maps between higher rank bounded symmetric domains
were studied by Kim and Seo~\cite{kim2022holomorphicity} under the additional
assumption that the map is $C^1$.
They showed that if the domain is irreducible, the map is $C^1$-smooth,
and the rank of the co-domain is no greater than the rank of the domain,
then the map is either holomorphic or anti-holomorphic.
There are examples where the irreducibility or rank conditions
do not hold and the conclusion fails.

We generalise these results to distance-preserving maps between bounded
symmetric domains of arbitrary rank,
without any smoothness condition on the map.
Such a condition was needed
in~\cite{kim2022holomorphicity} to apply techniques from differential geometry.
Here, instead we exploit the large-scale geometry of bounded symmetric domains
with their Carath\'eodory distance.
Keys tools are the horofunction boundary of this space,
which was described in detail in
Chu--Cueto-Avellaneda--Lemmens \cite{chu_cueto_avellaneda_lemmens},
and the Gromov product.

Recall that a \emph{Hermitian symmetric space} is a Riemannian symmetric space
with a compatible complex structure. It was shown by Harish-Chandra that every
non-compact type Hermitian symmetric space
can be realised as a \emph{bounded symmetric domain}.
These are the bounded domains in $\C^n$ such that every point in the domain
is an isolated fixed point of a biholomorphic involution from the domain
to itself. It was later shown \cite{Kaup1,Koecher} that every such domain
arises as the open unit ball of a \emph{JB*-triple}.
These triples provide a powerful tool for studying bounded symmetric domains,
one we will use extensively in this paper.

Our first result is a sharpening and  extension to the non-smooth case
of \cite[Theorem 1.1]{kim2022holomorphicity}.
In particular, it shows that there is no distance-preserving map from
a bounded symmetric domain to another one of lower rank.

\begin{theorem}
\label{thm:rank_inequality}
Let $D$ be a finite dimensional bounded symmetric domain with rank $r$ and genus $p$,
and  $D'$ be a finite dimensional bounded symmetric domain with rank $r'$ and genus $p'$. 
If $\phi\colon D \to D'$ is a Carath\'eodory distance-preserving map, then 
\[
r \le r'\mbox{\qquad and\qquad } rp -\dim D \leq r'p'-\dim D'.
\]
\end{theorem}

When the rank of the domain and co-domain are the same, one can say more.
Our next theorem states that, in this case, a distance-preserving map
respects factors, that is, each irreducible factor of the domain
only influences one irreducible factor of the co-domain.

If $D_1, \dots, D_m$ are sets and $I := \{i_1, \dots, i_k\}$
is a subset of $\{1, \dots, m\}$,
we use the notation $D_I$ to denote the Cartesian product
$D_{i_1} \times \dots \times D_{i_k}$, where the $i_j$ are taken in increasing
order.
Also, if $x := (x_1, \dots, x_m) \in D_1 \times \dots \times D_m$,
then $x_I : = (x_{i_1}, \dots, x_{i_k}) \in D_I$.
If $\indexfn$ is a map, then $\indexfn^{-1}(x)$ denotes the preimage
of a point $x$.

\begin{theorem}
\label{thm:respects_product}
Let $\phi\colon D \to D'$ be a Carath\'eodory distance-preserving map
between two finite-dimensional bounded symmetric domains of equal rank.
Let these domains be expressed in terms of their irreducible factors
as follows:
\[
D = D_1 \times \dots \times D_m \mbox{\quad and\quad }
D' = D'_1 \times \dots \times D'_n.
\]
Then, there exists a surjective map
$\indexfn\colon \{1, \dots, m\}\to \{1, \dots, n\}$,
and distance-preserving maps
$\phi_k$; $k\in\{1, \dots, n\}$,
with $\phi_k \colon D_{\indexfn^{-1}(k)} \to D'_k$ such that
\begin{align*}
\phi(x_1, \dots, x_m)
  = \big(\phi_1(x_{\indexfn^{-1}(1)}), \dots, \phi_n(x_{\indexfn^{-1}(n)})\big).
\end{align*}
\end{theorem}

Note that it follows that $D$ has at least as many irreducible
factors as $D'$.

\begin{theorem}
\label{thm:holomorphic_or_antiholomorphic}
Let $\phi\colon D \to D'$ be a Carath\'eodory distance-preserving map
between two finite-dimensional bounded symmetric domains of equal rank.
Assume that $D$ is irreducible. Then, $\phi$ is either holomorphic
or antiholomorphic.
\end{theorem}

In the case where $D$ and $D'$ have the same number of irreducible
factors, the component maps $\phi_k$ of Theorem~\ref{thm:respects_product}
map irreducible factors to irreducible factors,
and so Theorem~\ref{thm:holomorphic_or_antiholomorphic} can be applied
to obtain that each of them is either holomorphic or antiholomorphic.
In particular, this is true when $D'$ is identical to $D$.
In this case we get the following immediate corollary.

\begin{corollary}
If $D$ is a finite-dimensional irreducible bounded symmetric domain,
then $\mathrm{Aut}(D):=\{\phi\colon D\to D \mid \mbox{$\phi$ biholomorphic}\}$
is an index two subgroup of
$\mathrm{Isom}(D) :=
    \{\phi\colon D\to D
 \mid \mbox{$\phi$ is a bijective  Carath\'eodory distance-preserving map}\}$.  
\end{corollary}

Recalling that every bounded symmetric domain can be realised as the unit ball
of a JB*-triple, we can strengthen the conclusion of
Theorem~\ref{thm:holomorphic_or_antiholomorphic}
to get the following rigidity result.
A \emph{triple homomorphism} between JB*-triples is a complex-linear map
that preserves the triple product.

\begin{theorem}
\label{thm:triple_homomorphism}
Let $\phi\colon D \to D'$ be a Carath\'eodory distance-preserving map
between two finite-dimensional bounded symmetric domains of equal rank.
If $\phi$ is holomorphic and $\phi(0) = 0$,
then $\phi$  is the restriction to the unit ball of a triple homomorphism.
\end{theorem}

As mentioned earlier, we use ideas from metric geometry to prove the results.
In particular, we will we analyse the extension of the distance-preserving map
to the horofunction boundary and the way it behaves on the parts of this
boundary.
We will show, for any bounded symmetric domain with its Carath\'eodory
distance, that each part of the horofunction boundary with the detour metric
is isometric to a Hilbert metric space on a symmetric cone;
see Theorem~\ref{Hilbert}. This result, which
extends~\cite[Propositions 8.4 and 8.5]{chu_cueto_avellaneda_lemmens},
will be used to prove Theorem~\ref{thm:rank_inequality}.

\section{Bounded symmetric domains and JB*-triples}
We will recall some results for JB*-triples that are needed in this paper.
Most of them can be found
in~\cite{chu_Jordan_structures_in_geometry_and_analysis},
\cite{chu_bounded_symmetric_domains_in_banach_spaces},
and~\cite{loos_bounded_symmetric_domains_and_jordan_pairs}. Throughout the paper the bounded symmetric domains and the JB*-triples will be finite dimensional.

A JB*-triple is a complex Banach space $V$ endowed with a triple product
\[
\{\cdot,\cdot,\cdot\} \colon V\times V \times V \to V,
\]
satisfying the following axioms, for $a,b,x,y,z\in V$:
\begin{enumerate}
\item[(i)]
$\{\cdot,\cdot,\cdot\}$ is linear and symmetric in the outer variables,
and conjugate linear in the middle variable;
\item[(ii)]
$\bigl\{a,b,\{x,y,z\}\bigr\}
    = \bigl\{\{a,b,x\},y,z\bigr\}
     -\bigl\{x, \{b,a,y\},z\bigr\}
    + \bigl\{x,y,\{a,b,z\}\bigr\}$;
\item[(iii)]
The operator $a\smallsquare a := \{a, a, \cdot\,\}$ from $V$ to $V$
is Hermitian, and has non-negative spectrum;
\item[(iv)]
$\|a \smallsquare a\| = \| a \|^2$.
\end{enumerate}

The \emph{box operator} $a\smallsquare b \colon V\to V$ is defined by
\begin{equation*}
a \smallsquare b(x) := \{a,b,x\},
\qquad\text{for all $x\in V$}.
\end{equation*}
A \emph{tripotent} of a JB*-triple $V$ is an element $e$ such that
$\{e,e,e\} = e$.
Each tripotent $e$ induces a decomposition of $V$ into eigenspaces of the
box operator $e \smallsquare e$.
The eigenvalues of this operator lie in the set $\{0, 1/2, 1\}$. Let
\begin{equation*}
V_k(e)
  := \Big\{ x \in V \Big\arrowvert e\smallsquare e(x) = \frac {k} {2} x \Big\},
\qquad\text{for $k \in \{0, 1, 2\}$},
\end{equation*}
be the corresponding eigenspaces,
which are known as the \emph{Peirce $k$-spaces} of $e$.
We have the algebraic direct sum
\begin{equation*}
V = V_0(e) \oplus V_1(e) \oplus V_2(e).
\end{equation*}
This is the \emph{Peirce decomposition} associated to $e$.
We have the following \emph{Peirce calculus}:
\begin{equation*}
\big\{ V_{i}(e), V_{j}(e), V_{k}(e) \big\} \subseteq V_{i-j+k} (e),
\qquad\text{if $i-j+k \in \{ 0,1,2\}$},
\end{equation*}
and
\begin{equation*}
\big\{ V_{i}(e), V_{j}(e), V_{k}(e) \big\} = \{0\},
\qquad\text{otherwise}.
\end{equation*}
Moreover,
\begin{equation*}
\big\{ V_{2}(e), V_{0}(e), V \big\}
    = \big\{ V_{0}(e), V_{2}(e), V \big\}
    = \{0\}.
\end{equation*}

Each Peirce $k$-space $V_k(e)$ is the range of the \emph{Peirce $k$-projection}
$P_k(e) \colon V \to V$, defined by
\begin{equation*}
P_2(e) := Q_e^2,
\qquad
P_1(e) := 2(e \smallsquare e - Q_e^2),
\qquad
P_0(e) := B(e,e).
\end{equation*}
Here, $Q_a \colon V \to V$ is, for $a\in V$, the \emph{quadratic operator}
\begin{equation*}
Q_a(x) = \{a,x,a\},
\qquad\text{for all $x\in V$},
\end{equation*}
and $B(a,b)\colon V\to V$, with $a, b\in V$, is the \emph{Bergman operator}
\begin{equation}
\label{eqn:bergman_definition}
B(a,b)(x) := x - 2(a\smallsquare b)(x) + \big\{ a,\{b,x,b\},a \big\},
\qquad\text{for all $x\in V$}.
\end{equation}

For each $a\in D$, the \emph{M\"obius transformation} $g_a \colon D\to D$
is defined to be
\begin{equation*}
g_a(x) := a + B(a,a)^{1/2} (\id + x \smallsquare a)^{-1}(x),
\qquad\text{for all $x\in D$}.
\end{equation*}
Here $\id$ denotes the identity operator on $V$.
The inverse operator in this definition exists because
$\|x\smallsquare a\| \le \|x\| \|a\| < 1$.
Observe that, for each $a\in D$, the M\"obius transformation $g_a$
maps $0$ to $a$. Moreover, $g_a$ is a bijection from $D$ to itself, and its inverse is $g_{-a}$.

It can be shown that, for two elements $a$ and $b$ of $V$,
we have $a \smallsquare b =0$ if and only if $b \smallsquare a =0$.
In this case the two elements are said to be \emph{orthogonal}.
Another equivalent condition is that $\{a,a,b\} =0$.
For orthogonal elements $a$ and $b$,
\begin{equation*}
\|a+ b\|= \max \big\{ \|a\|, \|b\| \big\};
\end{equation*}
see~\cite[Corollary 3.1.21]{chu_Jordan_structures_in_geometry_and_analysis}.

One can define an ordering on the set of tripotents by writing
$c \le e$ if
$e = c + c'$, where $c'$ is a tripotent
orthogonal to $c$.
A tripotent is \emph{minimal} if it is non-zero and minimal with respect
to this ordering.
This is the case precisely when the tripotent $e$ satisfies
$V_2(e)=\mathbb{C}e$.

An orthogonal set of non-zero tripotents
is linearly independent, and every tripotent can be written as a sum of
orthogonal minimal tripotents. The maximum number of mutually orthogonal tripotents in $V$ is called the \emph{rank} of $V$.  
A \emph{frame} is a maximal orthogonal system of minimal tripotents.
The \emph{rank} of a tripotent $e$ is the rank of the sub-triple $V_2(e)$, and $e$ is said to be a \emph{maximal} tripotent if the rank of $V_2(e)$ is equal to the rank of $V$.
If $e = e_1 + \dots + e_s$ is a decomposition of $e$ into a sum of orthogonal minimal tripotents, then the rank of $e$ is $s$.

The {\em genus} $p$ of a JB*-triple $V$ with rank $r$ is defined as
\[
p:= \frac{2}{r}\mathrm{Trace} ( e\smallsquare e),
\]
where $e$ is a maximal tripotent.  It is known, see \cite[p.194]{chu_bounded_symmetric_domains_in_banach_spaces}, that 
\[
rp = \dim V_2(e) + \dim V,
\]
for any maximal tripotent $e\in V$.

Each family $e_{1}, \dots, e_{n}$ of orthogonal tripotents defines a
decomposition of $V$ as follows.
For $i,j \in \{0,1, \ldots, n\}$, define the \emph{joint Peirce space}
\begin{equation*}
V_{ij}(e_{1},\dots,e_{n})
    := \Big\{z \in V \mid
          \{e_{k}, e_{k}, z\} = \frac {1} {2} (\delta_{ik} + \delta_{jk}) z,
               \,\text{for all $k=1,\ldots,n$} \Big\}.
\end{equation*}
Here $\delta_{ij}$ is the Kronecker delta,
which equals $1$ if $i = j$ and is zero otherwise.
The \emph{joint Peirce decomposition} of $V$ is
\begin{equation*}
V = \bigoplus_{0 \le i \le j \leq n} V_{ij}.
\end{equation*}
The joint Peirce spaces satisfy the multiplication rules
\begin{align*}
\{ V_{ij}, V_{jk}, V_{kl}\} &\subseteq V_{il},
    \qquad\,\, \text{for all $i, j, k, l$},
    \quad \text{and} \\
    V_{ij} \smallsquare V_{kl} &=\{0\},
    \qquad \text{for $i,j \notin \{k, l\}$}.
\end{align*}
For each $i$ and $j$, there is a contractive projection
$P_{ij} (e_{1}, \dots, e_{n})$ from $V$ to $V_{ij} (e_{1}, \ldots, e_{n})$
called the \emph{joint Peirce projection}.
We will occasionally denote these simply by $P_{ij}$
if it is clear which tripotents $e_1, \ldots, e_n$ are involved.

\newcommand\iso{f}
\newcommand\mapclass{h}
\newcommand\distfn{\psi}
\newcommand\dist{d}
\newcommand\closure{\operatorname{cl}}
\newcommand{\Rplus}{\mathbb{R}_+}

\section{The horofunction boundary}

In the section we discuss the relevant metric geometry concepts. We will work in metric spaces $(X,\dist)$ that are proper, i.e., all closed balls are compact, and $(X,\dist)$ is geodesic, meaning that every pair of points can be connected
by a geodesic arc. 
Associate to each point $z\in X$ the function $\distfn_z\colon X\to \R$,
\begin{equation*}
\distfn_z(x) := \dist(x,z)-\dist(b,z),
\end{equation*}
where $b\in X$ is some base-point.
The map $\distfn\colon X\to \mathcal{C}(X),\, z\mapsto \distfn_z$
is injective and continuous.
Here, $\mathcal{C}(X)$ denotes the space of continuous
real-valued functions on $X$ with the topology of pointwise convergence.
The closure $\closure\distfn(X)$ is compact.
As $(X,\dist)$ is proper and geodesic, $\psi$ is a homeomorphism between $X$ and $\psi(X)$,
and hence $\closure\distfn(X)$ is a compactification of $X$.
We call it the \emph{horofunction compactification}.
As $(X,\dist)$ is proper, it is also separable, hence the topology of pointwise convergence on $\closure\distfn(X)$ is metrisable. This implies that each horofunction is the limit of a sequence of points in $X$.   

We define the \emph{horofunction boundary} of $(X,\dist)$ to be
\begin{equation*}
X(\infty) :=\big(\closure\distfn(X)\big)\backslash\distfn(X).
\end{equation*}
The elements of this set are the \emph{horofunctions} of $(X,\dist)$.
They may be thought of as ``points at infinity'' of the metric space.
The definition of the horofunction boundary is essentially 
due to Gromov~\cite{gromov:hyperbolicmanifolds},
although he used a different topology.

Although the definition appears to depend on the choice of base-point,
one can check that horofunction boundaries coming from different base-points
are homeomorphic, and that corresponding horofunctions differ only by an
additive constant.

\subsection{Busemann points and the detour cost}

\label{def:almost_almost_geodesic} A sequence $(z_n)$ in a metric space is called an~\emph{almost geodesic} if,
for all $\epsilon>0$,
\begin{equation*}
d(z_0,z_j) \ge d(z_0,z_i) + d(z_i,z_j) - \epsilon,
\end{equation*}
for $i$ and $j$ large enough, with $i\le j$.
This definition is similar to Rieffel's~\cite{rieffel_group},
except that here the almost geodesics are unparameterised.
Note that any subsequence of an almost geodesic is also an almost geodesic.

Rieffel~\cite{rieffel_group} showed that
every almost-geodesic converges to a limit in the horofunction boundary.
A horofunction is said to be a \emph{Busemann point}
if there is an almost-geodesic converging to it.

We define the \emph{detour cost}
for any two horofunctions $\xi$ and $\eta$ in $X(\infty)$ to be
\begin{equation*}
H(\xi,\eta)
    := \liminf_{x\to\xi} \big( \dist(b,x) + \eta(x) \big)
     = \sup_{V\in \mathcal{U}_\xi}
           \Big( \inf_{x\in V\cap X}  \dist(b,x) + \eta(x) \Big),
\end{equation*}
where $\mathcal{U}_\xi$ is the collection of neighbourhoods of $\xi$.
This concept appeared first in~\cite{AGW-m}.
Intuitively, it is an extension to the boundary of the excess of the triangle
inequality $\dist(b, x) + \dist(x, y) - \dist(b, y)$, where $y$ tends to $\eta$,
and $x$ tends to $\xi$. Thus, it measures the cost of taking a detour
close to $\xi$ on the way from $b$ to $\eta$.

In the case where $\xi$ is a Busemann point, it suffices to calculate
the limit along any almost-geodesic $(z_n)$ converging to it, that is,

\begin{equation*}
 H(\xi,\eta) = \lim_{n\to\infty} \big( \dist(b,z_n) + \eta(z_n) \big),
\end{equation*}
for any horofunction $\eta$;
see \cite[Lemma 2.6]{walsh_hilbert_and_thompson_metrics_isometric_to_infinite_dimensional_banach_spaces}.

The Busemann points can be characterised as follows:
a horofunction $\xi$ is Busemann if and only if $H(\xi,\xi)=0$.

The detour cost is non-negative and satisfies the triangle inequality,
but it is not necessarily symmetric and may take the value zero between
two distinct points.
We obtain better properties, however, when we symmetrise.
.
For Busemann points $\xi$ and $\eta$, define the \emph{detour metric}
\begin{equation*}
\delta(\xi,\eta) := H(\xi,\eta) +H (\eta,\xi).
\end{equation*}

This function is a (possibly $\infty$-valued) metric
on the set of Busemann points.
It is independent of the choice of basepoint.

We may consider a pair of Busemann points to be related if the distance between
them in the detour metric is finite. This is an equivalence relation,
and so partitions the set of Busemann points into what we call
\emph{parts}; these are the maximal subsets on which the detour metric
is a genuine metric.
When a part consists of a single Busemann point, we call that point a
\emph{singleton}. These are of particular interest because they tend to
be the simplest and most tractable horofunctions.

\section{The Gromov product on the horofunction boundary}

We define the Gromov product of a pair of points $x$ and $y$
with respect to a basepoint $b$ as follows:
\begin{equation*}
\gromprod {x} {y} {b} := d(x, b) + d(b, y) - d(x, y).
\end{equation*}
This product may be extended to the horofunction boundary:
\begin{equation*}
\gromprod {\xi} {\eta} {b}
    := \liminf_{x\to\xi,\, y\to \eta} \gromprod {x} {y} {b}
     = \sup_{V\in \mathcal{U}_\xi,W\in \mathcal{U}_\eta}
         \Big(\inf_{x\in V\cap X,\, y\in W\cap X} \gromprod {x} {y} {b} \Big),
\end{equation*}
for all $\xi$ and $\eta$ in $X(\infty)$,
where $\mathcal{U}_\xi$ and $\mathcal{U}_\eta$ are the collections
of neighbourhoods of $\xi$ and $\eta$, respectively. 

It is similar to the detour cost in that it is an extension to the boundary
of the excess of the triangle inequality;
however, the choice of the two points that go to infinity is different.
Like the detour cost, it is invariant under distance-preserving maps,
provided the basepoint is mapped to the basepoint.
The two quantities provide complementary information.
The Gromov product can distinguish for example between hyperbolic space
and Euclidean space:
in the former, it is finite for all pairs of distinct points,
whereas in the latter it is infinite for all pairs that are not opposite
one another.
In contrast, in both spaces, the detour cost is always infinite
for pairs of distinct points.

When the horofunctions are Busemann points,
we have the following alternative expressions for the Gromov product.
Recall that a sequence of real-valued functions $f_n\colon X\to\mathbb{R}$
is \emph{almost non-increasing} if,
for every $\epsilon>0$, there exists $N$ such that
$f_j(x) \le f_i(x) + \epsilon$, for all $N\le i\le j$ and all $x\in X$.
Such sequences are closely related to almost geodesics,
as the following proposition shows.

\begin{proposition}
[\cite{walsh_hilbert_and_thompson_metrics_isometric_to_infinite_dimensional_banach_spaces}]
\label{prop:almost_non_inc_busemann}
A sequence $(z_n)$ in a metric space $(X,d)$ is an almost geodesic
if and only if $\psi_{z_n}(\cdot) := d(\cdot,z_n)-d(b,z_n)$ is an
almost non-increasing sequence.
\end{proposition}

\begin{proposition}
\label{prop:alt_formula_for_gromov_prod}
Let $(x_i)$ and $(y_j)$ be two almost-geodesic sequences
in a metric space $(X, d)$, converging respectively to Busemann points
$\xi$ and $\eta$. Then,
\begin{equation*}
\gromprod {\xi} {\eta} {b}
    = -\inf_{z\in X} \big[\xi(z) + \eta(z)\big]
    = \lim_{i,j\to\infty}\,
          \gromprod {x_i} {y_j} {b}.
\end{equation*}
\end{proposition}

\begin{proof}
First we note that for each $x$ and $y$ in $X$,
we have $d(x, y) = \inf_z [d(x, z) + d(z, y)]$. Thus, 
\begin{equation}
\label{eqn:finite_gromov_prod_as_infimum}
\gromprod {x} {y} {b}
    = -\inf_{z\in X} \big[d(x, z) + d(z, y) - d(x, b) - d(b, y)\big].
\end{equation}

Let $\epsilon>0$ be given, and let $z^*\in X$ be such that 
\begin{equation}
\label{eqn:almost_optimal_x}
\xi(z^*)+\eta(z^*)<\inf_{z\in X}\big[\xi(z)+\eta(z)\big] + \epsilon.
\end{equation}
Taking $V\in\mathcal{U}_\xi$ and $W\in\mathcal{U}_\eta$ such that 
\begin{align*}
d(x, z^*) - d(x,b) &< \xi(z^*)+\epsilon,
\qquad\mbox{for all $x\in V \intersection X$}, \\
\qquad\text{and}\quad
d(z^*, y) - d(b, y) &< \eta(z^*)+\epsilon,
\qquad\mbox{for all $y\in W \intersection X$},
\end{align*}
we find for $x\in V \intersection X$ and $y\in W \intersection X$ that 
\begin{align*}
\inf_{z\in X} \big[ d(x, z) + d(z, y) - d(x, b) - d(b, y) \big]
    &\leq d(x, z^*) + d(z^*, y) - d(x, b) - d(b, y)\\ 
    &<    \xi(z^*)+\eta(z^*)+2\epsilon\\
    &<    \inf_{z\in X}\big[\xi(z)+\eta(z)\big] + 3\epsilon.
\end{align*}
Using~\eqref{eqn:finite_gromov_prod_as_infimum}, we get
\begin{equation*}
\inf_{x\in V\cap X,\, y\in W\cap X} \gromprod {x} {y} {b}
    \geq -\inf_{z\in X}\big[\xi(z)+\eta(z)\big] -3\epsilon,
\end{equation*}
and so
$\gromprod {\xi} {\eta} {b} \geq -\inf_{z\in X} \big[\xi(z) + \eta(z)\big]$.

Next we show that the second expression is no less than the third.
Note that for each $z\in X$ there exists $N_z\geq 1$ such that 
\[
\psi_{x_i}(z)\geq \xi(z)-\epsilon
\qquad\mbox{and}\qquad
\psi_{y_j}(z)\geq \eta(z)-\epsilon,
\]
for all $i,j\geq N_z$. Thus,
\[
\psi_{x_i}(z)+\psi_{y_j}(z)
    \geq \xi(z)+\eta(z) -2 \epsilon
    \geq \inf_{x\in X} \big[\xi(x)+\eta(x)\big] -2\epsilon,
\]
for all $i,j\geq N_z$.
As $\psi_{x_i}$ and $\psi_{y_j}$ are both almost non-increasing sequences
of functions, there exists an $N_1\geq 1$ such that,
\[
\psi_{x_k}(x) -\epsilon\leq \psi_{x_i}(x)
\qquad\mbox{and}\qquad
\psi_{y_l}(x) -\epsilon\leq \psi_{y_j}(x),
\]
for all $x\in X$, and for all $k\geq i\geq N_1$ and $l\geq j\geq N_1$. 

For each $i,j\geq N_1$, let $z_{ij}\in X$ be such that 
\[
\inf_{x\in X} \big[\psi_{x_i}(x) +\psi_{y_j}(x)\big]
    \geq \psi_{x_i}(z_{ij})+\psi_{y_j}(z_{ij}) -\epsilon.
\]
Using the previous inequalities, we find that
\[
\inf_{x\in X}\big[\psi_{x_i}(x) +\psi_{y_j}(x)\big]
    \geq \psi_{x_k}(z_{ij})+\psi_{y_l}(z_{ij}) -3\epsilon,
\]
for each $k\geq i\geq N_1$ and each $l\geq j\geq N_1$. 

Thus, for each $k\geq i\geq N_1$ and each $l\geq j\geq N_1$ with $k,l\geq N_{z_{ij}}$, we have
\[
\inf_{x\in X}\big[\psi_{x_i}(x) +\psi_{y_j}(x)\big]
    \geq  \inf_{x\in X}\,\big[\xi(x)+\eta(x)\big] -5\epsilon.
\]
This implies that 
\[
\limsup_{i,j\to\infty}\, \gromprod {x_i} {y_j} {b}
    = \limsup_{i,j\to\infty}
          \Big(-\inf_{x\in X} \big[\psi_{x_i}(x) + \psi_{y_j}(x)\big]\Big)
    \leq -\inf_{x\in X}\, \big[ \xi(x)+\eta(x) \big].
\]

To complete the proof, observe that
\begin{equation*}
\gromprod {\xi} {\eta} {b}
    = \liminf_{x\to\xi,\, y\to \eta} \gromprod {x} {y} {b}
    \leq \liminf_{i,j\to\infty} \gromprod {x_i} {y_j} {b}.
\qedhere
\end{equation*}
\end{proof}

\section{distance-preserving maps and the horofunction boundary}
\label{sec:isometric_embeddings}

\begin{figure}
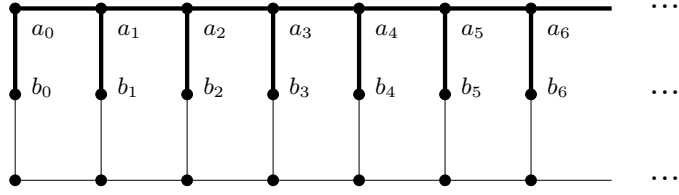
\caption{A metric space (bold) embedded in another (bold and non-bold).
Two sequences, along the top and middle respectively, converge to the same
horofunction in the smaller space, but not in the larger one.}
\label{fig:halfladder}
\end{figure}

When there is a {\em surjective} distance-preserving map between
two metric spaces, the map extends to a homeomorphism between
their horofunction compactifications.
In the absence of surjectivity, however, the situation is more complicated,
as the map does not necessarily extend continuously to a map between the boundaries.

Consider for example the metric spaces depicted in Figure~\ref{fig:halfladder}.
The space in bold is isometrically embedded in a larger space.
The metric in both cases is the \emph{path length metric},
where the distance between two points is the Euclidean length
of the shortest path joining the two points
that remains within the metric space.
The smaller space has only one horofunction,
and both sequences $(a_n)$, along the top, and $(b_n)$, along the middle,
converge to it.
On the contrary, the same sequences in the larger space converge
to different points.
Note that $(a_n)$ is an almost-geodesic (in fact a geodesic) in both spaces,
while $(b_n)$ is not an almost-geodesic in either space.

When we restrict our attention to almost-geodesic sequences,
the situation becomes more satisfying.
Of course, an almost-geodesic is mapped by a distance-preserving map 
to an almost-geodesic.
We shall see in addition that, if two almost-geodesics converge to the same
Busemann point in the domain, then their images converge to the same point.

\begin{definition}
A map $\phi \colon X \to Y$ between two metric spaces $(X, d_X)$ and $(Y, d_Y)$
is \emph{distance-preserving} if $d_Y(\phi(x), \phi(y)) = d_X(x, y)$,
for all $x, y\in X$.
\end{definition}

Recall the following~\cite[Lemma 3.18]{haettel_schilling_walsh_wienhart}.

\begin{proposition}
\label{prop:back_and_forth}
Two almost-geodesic sequences $(x_n)$ and $(y_n)$ in a metric space
converge to the same Busemann point if and only if
there exists an almost-geodesic sequence $(z_n)$
that has infinitely many points in common with both $(x_n)$ and $(y_n)$.
\end{proposition}

\begin{corollary}
\label{prop:mapping_of_busemann_points}
Let $\phi \colon X \to Y$ be a distance-preserving map between metric spaces.
Then, there exists an injective map, which we also denote by $\phi$,
from the set of Busemann points of $X$ to those of $Y$,
with the property that every almost-geodesic $(x_n)$
converging to a Busemann point $\xi$ gets mapped to an almost-geodesic
$(\phi(x_n))$ converging to the Busemann point $\phi (\xi)$.
\end{corollary}
\begin{proof}
Given any Busemann point $\xi$ of $X$, take an almost-geodesic $(x_n)$
converging to it, and define $\phi (\xi)$ to be the limit of $\phi(x_n)$.
That this is independent of the almost-geodesic chosen follows from
Proposition~\ref{prop:back_and_forth}.
For the injectivity, we take two almost-geodesics $(y_n)$ and $(z_n)$
in $\phi(X)$ converging to the same Busemann point.
We then apply Proposition~\ref{prop:back_and_forth} to get an almost-geodesic
that has infinitely many points in common with both; this sequence may
furthermore be chosen to consist entirely of points of $(y_n)$ and $(z_n)$.
As such, it has a preimage, which is also an almost-geodesic and so converges
to a Busemann point of $X$.
This shows that the preimages of $(y_n)$ and $(z_n)$ have the same limit.
\end{proof}

As well as this, distance-preserving maps preserve the detour cost
of every pair of Busemann points and  their Gromov product,
assuming that the basepoint is mapped to the basepoint.
The latter statement follows from
Proposition~\ref{prop:alt_formula_for_gromov_prod}.

\section{The horofunction boundary of bounded symmetric domains}

Chu--Cueto-Avellaneda--Lemmens~\cite{chu_cueto_avellaneda_lemmens}
have determined the horofunction boundary of a finite dimensional bounded symmetric domain $D$
under the Carath\'eodory distance. 
Recall that the Carath\'eodory  distance on a domain
$\Omega\subseteq \mathbb{C}^n$ is given by 
\[
d(x,y) := \sup \Big\{ \omega \big (f(x),f(y) \big) \mid  f \in H(\Omega, \mathbb{D}) \Big\},
\qquad \text{for all $x,y\in \Omega$},
\]
where $H(\Omega,\mathbb{D})$ is the set of all holomorphic functions
$f\colon  \Omega\longrightarrow \mathbb{D}$
and $\omega$ is the hyperbolic distance on the unit disc
$\mathbb{D}:=\{z\in \mathbb{C}\mid |z|<1\}$, namely 
\[
\omega(z,w) := \tanh^{-1}\left( \left|\frac{z-w}{1-z\overline{w}}\right| \right) = \frac{1}{2}\log\left(\frac{ 1+ \left |\frac{z-w}{1-z\overline{w}}\right |}{1- \left|\frac{z-w}{1-z\overline{w}}\right|}\right).
\]
In the case of a bounded symmetric domain $D$, we have the following formula
(see~\cite[Theorem 3.5.9]{chu_Jordan_structures_in_geometry_and_analysis}):
\[
d(x,y) = \tanh^{-1}\|g_{-y}(x)\|,
\qquad\text{for all $x,y\in D$}.
\]

  .
It is known that, for bounded symmetric domains,
all the horofunctions are Busemann points and take the following form.

\begin{theorem}[Chu--Cueto-Avellaneda--Lemmens]
\label{thm:busemann_points_of_bsd}
Let $D$ be a finite dimensional bounded symmetric domain represented as the open unit ball of a
JB*-triple of rank $r$.
Every horofunction $\xi$ is a Busemann point and is of the form
\begin{equation}\label{xi}
\xi(z)
    = \frac {1} {2} \log \Big\lVert
      \sum_{1\le i\le j\le p}
          \lambda_i \lambda_j B(z, z)^{-1/2} B(z, e) P_{ij}
      \Big\rVert,
\qquad\text{for all $z\in D$},
\end{equation}
where $p \in \{1, \dots, r\}$, $0<\lambda_1, \dots, \lambda_p\leq 1$
with $\max_i \lambda_i = 1$,
and $e_1, \dots, e_p$ are mutually orthogonal nonzero minimal tripotents
with $e := e_1+\cdots +e_p\in\partial D$
and the $P_{ij} \colon V \to V$ are the corresponding joint Peirce projections.
\end{theorem}

In \cite[Proposition 8.5]{chu_cueto_avellaneda_lemmens},
it was shown that two horofunctions $\xi$ and $\eta$ with tripotents
$e$ and $c$, respectively, are in the same part if and only if $e=c$.
In this section, we will give an explicit formula for the detour cost
and the detour metric. In fact, we shall show that each part
with the detour metric is isometric to a Hilbert metric space
on a symmetric cone.  

Recall from~\cite[3.13]{loos_bounded_symmetric_domains_and_jordan_pairs} that,
given a tripotent $e\in V$, its Peirce 2-space $V_2(e)$ is a JB*-algebra
with product $x\bullet y = \{x,e,y\}$, involution $x^* =\{e,x,e\}$,
and unit $e$.
Moreover, if we let $A=A(e):= \{x\in V_2(e) \mid \{e,x,e\} =x\}$
be the self-adjoint part of $V_2(e)$,
then $A$ is a real closed subalgebra of the JB*-algebra $V_2(e)$
that forms a JB-algebra with cone of squares $A_+=\{x^2\colon x\in A\}$,
and $V_2(e)=A+\mathrm{i}A$. 

We write $D_e := V_2(e) \cap D$ and denote the \emph{Cayley transform}
by $c_e\colon D_e \to A+\mathrm{i}A_+^\circ$, so 
\[
c_e(z) := \mathrm{i}(e+z)(e-z)^{-1},
\qquad \text{for all $z\in D_e$}.
\]
The Cayley transform is a biholomorphic map;
see \cite[Example 2.4.18 and Section 3.6]
{chu_bounded_symmetric_domains_in_banach_spaces}, 
and hence a Carath\'eodory distance isometry,
and it maps $D_e\cap A$ onto $\mathrm{i}A_+^\circ$.
Here, $A_+^\circ$ denotes the interior of the cone  $A_+$,
and $a^{-1}$ is the (unique) element in the JB*-algebra $V_2(e)$
such that $a\bullet a^{-1}= e$ and $a^2\bullet a^{-1} = a$. 

On the open cone $A_+^\circ$ in the JB-algebra $A$,
there is a natural metric $d_T$ known as the {\em Thompson metric},
which is defined in terms of the partial order on $A$ induced by $A_+$,
namely, $x\leq y$ if $y-x\in A_+$.
More specifically, for $x\in A_+$ and $y\in A_+^\circ$, let 
\[
M(x/y) := \inf\{\lambda>0\mid x\leq \lambda y\}
       =\sup\left\{ \frac {\rho(x)} {\rho(y)}
            \mid \text{$\rho\in A^*$ with $\rho(e)=1$} \right\},   
\]
and define
\[
d_T(x,y) := \max \big\{\log M(x/y), \log M(y/x) \big\},
\qquad \text{for all $x,y\in A_+^\circ$}.
\]
It was shown by Vesentini~\cite{Ve} that
the Thompson distance on $A_+^\circ$
is related to
the restriction of the Carath\'eodory distance $d$ on
$A+\mathrm{i}A_+^\circ$ to $\mathrm{i}A_+^\circ$
as follows: 
\begin{equation}
\label{Ves}
d_T(x,y) = 2d(\mathrm{i}x,\mathrm{i}y),
\qquad \text{for all $x,y\in A_+^\circ$}.
\end{equation}
(Note that the factor 2 does not appear in~\cite{Ve}, as the factor $1/2$ in the hyperbolic metric on $\mathbb{D}$
in the definition of the Carath\'eodory distance is omitted there.) 

Let $\xi$ and $\eta$ be two horofunctions, with $\xi$ given by \eqref{xi}
and $\eta$ given by 
\begin{equation}
\label{eta}
\eta(z)
    = \frac {1} {2} \log \Big\lVert
      \sum_{1\le i\le j\le q}
          \mu_i \mu_j B(z, z)^{-1/2} B(z, c) P'_{ij}
      \Big\rVert,
\qquad\text{for all $z\in D$},
\end{equation}
where $q \in \{1, \dots, r\}$, $c=c_1+\cdots +c_p$ is a nonzero tripotent,
the $P'_{ij} \colon V \to V$ are the joint Peirce projections induced by
the mutually orthogonal minimal tripotents $c_1, \dots, c_q$, 
and $0<\mu_1, \dots, \mu_q\leq 1$ with $\max_i \mu_i = 1$.
Our next lemma, which determines the detour cost $H(\xi,\eta)$,
is an extension
of~\cite[Propositions 8.4 and 8.5]{chu_cueto_avellaneda_lemmens}.
Recall that if $c\leq e$, then  $e-c$ is a tripotent
and $(e-c)\smallsquare c =0$.

\begin{lemma}
\label{lem:relate_order_and_finiteness_of_detour}
Let $\xi$ and $\eta$ be horofunctions of a bounded symmetric domain,
where $\xi$ is given by \eqref{xi} and $\eta$ is given by \eqref{eta}. 
Then, $H(\xi, \eta) < \infty$ if and only if $c \le e$.
Moreover, if $c \le e$, then 
\begin{equation}
\label{H}
H(\xi,\eta) = \frac{1}{2}\log M(b/a),
\end{equation}
where $a := \sum_{i=1}^p \lambda_i^2e_i\in A_+^\circ$
and  $b := \sum_{i=1}^q \mu_i^2 c_i\in A_+$.
\end{lemma}

\begin{proof}
The forward implication appears in the proof of~\cite[Proposition 8.4]{chu_cueto_avellaneda_lemmens}.
To prove the backward implication, it suffices to show~\eqref{H}.

We know from~\cite[Lemma~8.3]{chu_cueto_avellaneda_lemmens} that the  path
\begin{equation*}
\gamma(t) := \sum_{i=1}^p \tanh(t - \alpha_i) e_i,
\end{equation*}
where  $\alpha_i = -\log \lambda_i$,
is a geodesic converging to $\xi$. Likewise, 
\begin{equation*}
\gamma'(t) := \tanh(0)(e-c) + \sum_{i=1}^q \tanh(t - \beta_i) c_i,
\end{equation*}
where  $\beta_i = -\log \mu_i$, is a geodesic converging to $\eta$. 

Using the fact that
$\tanh^{-1}(s) = \frac{1}{2}\log \left(\frac{1+s}{1-s}\right)$ for $-1<s<1$,
we get 
\[
c_e \big(\gamma'(t)\big)
    = \mathrm{i}\Big((e-c)+\sum_{i=1}^q e^{2t-2\beta_i}c_i \Big)
    \in \mathrm{i}A_+^\circ,
\qquad \text{for all $t>0$},
\]
which is a geodesic of the Carath\'eodory distance.
So, by \eqref{Ves} we see that 
\[
\sigma(t):= (e-c) +\sum_{i=1}^q e^{2t-2\beta_i}c_i 
\]
is a geodesic in $A_+^\circ$ with respect to $\frac{1}{2}d_T$.
It follows that 
\[
\psi_{\sigma(t)}(x)
    =\frac{1}{2}d_T \big(x,\sigma(t) \big)
         -\frac{1}{2}d_T \big(e, \sigma(t) \big)
\]
converges, as $t$ tends to infinity, to a horofunction, say $\eta_T$,
in $(A_+^\circ,\frac{1}{2}d_T)$.
These horofunctions have been analysed in~\cite{Le2}.
Following the proof of~\cite[Theorem 3.2]{Le2}, we find that 
\[
\eta_T(x) = \frac{1}{2}\log M(b/x),
\qquad \text{for all $x\in A_+^\circ$},
\]
where $b= \sum_{i=1}^q e^{-2\beta_i}c_i = \sum_{i=1}^q \mu_i^2 c_i$.
Indeed,
\[
\hat{\sigma}(t)
    := \frac {\sigma(t)} {e^{2t}}
    \to  \sum_{i=1}^q e^{-2\beta_i}c_i
    = b
\quad \text{and} \quad
\hat{\rho}(t)
     :=\frac {\sigma(t)^{-1}} {e^{2t}}
     \to 0,
\qquad
\text{as $t\to\infty$}.
\]
Thus, for $x\in A_+^\circ$, 
\begin{align*}
\psi_{\sigma(t)}(x)
    &= \frac{1}{2} \Big(d_T \big(x,\sigma(t) \big)- 2t \Big) \\
    &= \frac{1}{2} \max \Big\{ \log M \big(\hat{\sigma}(t)/x \big),
                          \log M \big( \hat{\rho}(t)/x^{-1} \big)\Big\} \\
    &\to \frac{1}{2}\log M(b/x).
\end{align*}
Here, we are using the fact that $M(x/y) = M(y^{-1}/x^{-1})$,
for all $x,y\in A_+^\circ$;
see for instance~\cite[p.~1518]{LRW}. 

Write $c_e \big(\gamma(t)\big) =: \mathrm{i}\tau(t)$, for $t>0$. 
So,
\[
\tau(t)= \sum_{i=1}^p e^{2t-2\alpha_i}e_i
\qquad\text{and} \quad
e^{-2t}\tau(t)\to \sum_{i=1}^p e^{-2\alpha_i} e_i = a.
\]
The function $x\in A^\circ_+\mapsto M(y/x)$ is continuous on $A_+^\circ$,
for each $y\in A_+$. Thus, 
\[
d\big(0, \gamma(t) \big) + \eta \big(\gamma(t)\big)
    = t + \frac{1}{2}\log M \big(b/\tau(t) \big)
    = \frac{1}{2}\log M \Big(b/\big(e^{-2t}\tau(t)\big) \Big)
    \to \frac{1}{2}\log M(b/a),
\]
as $t$ tends to infinity, and hence~\eqref{H} holds.
Note that $M(b/a)<\infty$, since $a\in A_+^\circ$ and $b\in A_+$.  
\end{proof}

On $A_+^\circ$ we also have the \emph{Hilbert metric}, 
\[
d_H(x,y) := \frac{1}{2}\log \big(M(x/y)M(y/x)\big),
\qquad \text{for all $x,y\in A_+^\circ$}.
\]
This is a metric between pairs of rays in $A_+^\circ$,
as $d_H(\lambda x,\mu y) = d_H(x,y)$,
for all $\lambda,\mu>0$ and $x,y\in A_+^\circ$. 
So, $d_H$ is a genuine metric on the cross-section
$\Sigma_e:=\{x\in A_+^\circ\mid M(x/e)=1\}$. 
As a consequence of the previous lemma,
each part of the horofunction boundary
of $(D,d)$ is isometric to a Hilbert metric space on symmetric cone.

\begin{theorem}
\label{Hilbert}
If $\mathcal {P}_\xi$ is the part containing the Busemann point $\xi$,
where $\xi$ is given by~\eqref{xi} with tripotent $e$,
then $(\mathcal{P}_\xi,\delta)$ is isometric to $(\Sigma_e, d_H)$. 
\end{theorem}

\begin{proof}
Let $\eta\in P_\xi$ be given by \eqref{eta}.
Since $\xi$ and $\eta$ are in the same part, we get from
Lemma~\ref{lem:relate_order_and_finiteness_of_detour} that $c=e$.
Defining $b:=\sum_i \mu_i^2 c_i$, we have that $M(b/e)=1$,
since $\max_i \mu_i =1$.
So the map defined in this way that sends each $\eta \mapsto b$
is a surjective isometry
between $(\mathcal{P}_e,\delta)$ and $(\Sigma_e,d_H)$,
using Lemma~\ref{lem:relate_order_and_finiteness_of_detour} again.
\end{proof}

The singleton Busemann points, that is, the ones having no other point in their part, have a particularly simple form.

\begin{proposition}
\label{prop:singletons_of_bsd}
The singleton Busemann points of a bounded symmetric domain $D$
are the functions of the form
\begin{equation*}
\xi(z) = \frac {1} {2} \log \big\| B(z, z)^{-1/2} B(z, e) e \big\|,
\qquad\text{for all $z\in D$},
\end{equation*}
where $e$ is a minimal tripotent.
\end{proposition}

\begin{proof}
From Theorems~\ref{thm:busemann_points_of_bsd}
and~\ref{Hilbert},
we see that the singletons are exactly the functions of the form
\begin{equation*}
\xi(z) = \frac {1} {2} \log \big\| B(z, z)^{-1/2} B(z, e) P_{2}(e) \big\|,
\qquad\text{for all $z\in D$},
\end{equation*}
where $e$ is a minimal tripotent.
Here we are using that $P_{11}(e) = P_2(e)$.
Recall that $P_2(e)$ is a contractive mapping,
that is, $\|P_2(e) x\| \le \|x\|$ for all $x\in V$.
Hence the supremum of
\begin{equation*}
\frac {\|B(z, z)^{-1/2} B(z, e) P_2(e) x\|} {\|x\|} \\
\end{equation*}
is attained in $V_2(e) \backslash\{0\}$.
But each $x\in V_2(e) \backslash\{0\}$ can be written $x = \lambda e$ for some
$\lambda \in \C \backslash\{0\}$, and so the supremum is equal to
$\| B(z, z)^{-1/2} B(z, e) e \|$.
\end{proof}

So, there is a one-to-one correspondence between the minimal tripotents
and the singleton Busemann points. This correspondence will play a key role
in what follows.

\section{Parts of the boundary, and tripotents}

In this section, we prove Theorem~\ref{thm:rank_inequality}. 
We also show that, if the ranks are equal, then a distance-preserving map 
takes singleton Busemann points to singleton Busemann points.
Throughout the section, $D$ is a bounded symmetric domain represented
as the unit ball of a finite-dimensional JB*-triple.

The following lemma is stated
in~\cite[Result 2.4]{bharali_janardhanan_proper_holomorphic_maps_between_bounded_symmetric_domains_revisited},
although no proof or reference is given.

\begin{lemma}
\label{lem:orthogonal_to_parts}
Let $e_1$ and $e_2$ be orthogonal tripotents of $D$,
and let $u\in D$ be orthogonal to $e = e_1 + e_2$.
Then, $u$ is orthogonal to both $e_1$ and $e_2$.
\end{lemma}

\begin{proof}
The orthogonality of $u$ and $e$ is equivalent to $u\in V_0(e)$.
Observe that
\begin{equation*}
\{e, e, e_1\}
    = \{e_1, e, e\}
    = \{e_1, e_1, e_1\} + \{e_1, e_1, e_2\} + \{e_1, e_2, e\}.
\end{equation*}
The first term on the right-hand-side equals $e_1$; the other two terms
are zero because $e_1$ and $e_2$ are orthogonal.
So, we see that $e_1$ is in $V_2(e)$.
Therefore, by the Peirce calculus, $\{u, e_1, e_1\} = 0$, which implies that
$u$ and $e_1$ are orthogonal. That the same holds for $u$ and $e_2$
can be proved similarly.
\end{proof}

Recall that a \emph{chain} in a partially ordered set is a subset
that is totally ordered in the sense that every pair $x$, $y$ of its elements
satisfies either $x \le y$ or $y \le x$.
A chain is \emph{maximal} if it is not a subset of any other chain.

\begin{lemma}
\label{lem:intervals_of_a_chain_are_orthogonal_tripotents}
Let $e_1, \dots, e_s$ be the elements of a chain of non-zero tripotents,
written in ascending order, and take $e_0 := 0$.
Then, the differences $e_{i} - e_{i-1};$ $i\in\{1, \dots, s\}$,
form an orthogonal family of tripotents.
\end{lemma}

\begin{proof}
Since $e_1 \le e_2$, we have that $e_2 - e_1$ is a tripotent orthogonal to
$e_1$. Similarly, $e_3 - e_2$ is a tripotent orthogonal to $e_2$,
and hence to both $e_2 - e_1$ and $e_1$,
by Lemma~\ref{lem:orthogonal_to_parts}.
We continue inductively to get the result.
\end{proof}

\begin{lemma}
\label{lem:the_rank_of_a_tripotent_is_its_index_in_every_maximal_chain}
Let $e_1, \dots, e_s$ be the elements of a maximal chain, written in ascending
order. Then, the rank of $e_i$ is $i$.
\end{lemma}

\begin{proof}
Fix $j \in \{1, \dots, i-1\}$. We can write $e_{j+1} = e_j + c$,
where $c$ is a non-zero tripotent orthogonal to $e_j$.
If $c$ were not minimal, then we could decompose it as $c = c_1 + c_2$,
into the sum of two non-zero orthogonal tripotents.
By Lemma~\ref{lem:orthogonal_to_parts},
$e_j$ would then be orthogonal to both $c_1$ and $c_2$,
and it would follow that $e_j \le e_j + c_1 \le e_j + c_1 + c_2 =  e_{j+1}$.
Thus $e_j + c_1$ could be added to the chain, contradicting its maximality.
So, $c$ must be minimal.

We conclude that each $e_i$ can be written as the sum of the minimal tripotents
$e_j - e_{j-1}$; $j\in\{1, \dots, i\}$. Here, we are taking $e_0 := 0$.
These tripotents are orthogonal,
using Lemma~\ref{lem:intervals_of_a_chain_are_orthogonal_tripotents}.
The conclusion follows.
\end{proof}

Every set of orthogonal tripotents is linearly independent.
So, for any finite-dimensional bounded symmetric domain, there is
an upper bound on the number of mutually orthogonal tripotents.
By Lemma~\ref{lem:intervals_of_a_chain_are_orthogonal_tripotents} then,
there is the same bound on the number of elements of a chain.
Therefore every chain is a subset of a maximal chain.

\begin{lemma}
\label{lem:sequence_of_tripotents}
Let $V$ be a JB*-triple of rank $r$. A chain of non-zero tripotents is maximal if and only if it contains exactly
$r$ elements. 
\end{lemma}

\begin{proof}
Let the elements of a maximal chain be $e_1, \dots, e_s$,
written in increasing order.
By Lemma~\ref{lem:the_rank_of_a_tripotent_is_its_index_in_every_maximal_chain},
the rank of $e_s$ is $s$.
But $e_s$ is a maximal tripotent, for otherwise
we could add another tripotent to the end of the sequence.
We conclude that $s=r$.

Now suppose that a chain $e_1, \dots, e_r$ has $r$ elements.
If it were contained in a larger chain, then the larger chain would itself
be contained in a maximal chain having strictly more than $r$ elements,
which is impossible by the first part.
\end{proof}

\begin{proof}[Proof of Theorem~\ref{thm:rank_inequality}]
We have seen in Section~\ref{sec:isometric_embeddings} that
the map $\phi$ induces a map, which we again denote by $\phi$,
from the set of Busemann points of $D$ to those of $D'$
in such a way that the image under $\phi$ of every almost geodesic converging
to a Busemann point $\xi$ of $D$ converges to $\phi (\xi)$.
This map preserves the detour cost in the sense that
$H(\phi( \xi), \phi (\eta)) = H(\xi, \eta)$ for all Busemann points $\xi$ and $\eta$
of $D$.
Therefore, two Busemann points lie in the same part if and only if
their images lie in the same part.
Since there is a one-to-one correspondence between parts of the
horofunction boundary and tripotents, we get a map,
again denoted by $\phi$, from the tripotents of $D$ to those of $D'$.
By Lemma~\ref{lem:relate_order_and_finiteness_of_detour},
the order on the set of tripotents is preserved; in fact,
$e \le c$ for two tripotents if and only if $\phi(e) \le \phi(c)$,
for any two tripotents $e$ and $c$ of $D$.
In particular, the map $\phi$ on the set of tripotents is injective.

Let $e_1, \dots, e_r$ be a maximal chain of tripotents in $D$.
Its image $\phi(e_1), \dots, \phi(e_r)$ is contained in a maximal chain
of tripotents of $D'$. So, by Lemma~\ref{lem:sequence_of_tripotents},
$\rank D = r \le \rank D'$.

To prove the second inequality,
we note that if $e$ is a maximal tripotent in $V$, then $\dim V_2(e) = rp -\dim V = rp-\dim D$; see \cite[p.194]{chu_bounded_symmetric_domains_in_banach_spaces}. As $V_2(e) = A+\mathrm{i}A_+^\circ$, we have that the dimension of the real vector space $A$ is equal to the dimension of the complex vector space $V_2(e)$.  By Theorem \ref{thm:busemann_points_of_bsd},
there exists a horofunction $\xi$ with tripotent $e$. The part $(\mathcal{P}_\xi,\delta)$ is isometric to the Hilbert metric space $(\Sigma_e,d_H)$ by Theorem \ref{Hilbert}.  Consider the horofunction $\phi(\xi)$ with tripotent, say $c$, in $V'$. Then, $\phi$ induces a distance-preserving map from $(\Sigma_e,d_H)$ into $(\Sigma_c,d_H)$. As $\dim \Sigma_e = \dim V_2(e) -1$ and  $ \dim \Sigma_c = \dim V_2(c) -1$, it follows from the invariance of domain theorem that $\dim V_2(e)\leq \dim V_2(c)$.  If $c'$ is a maximal tripotent  in $V'$, then $\dim V_2(c)\leq \dim V_2(c')$, and hence
\[
rp-\dim D
     = rp-\dim V
     = \dim V_2(e)
     \leq \dim V_2(c')
     = r'p' -\dim V'
     = r'p' -\dim D'.
\qedhere
\]

\end{proof}

The next lemma will be crucial to studying the equal rank case.

\begin{lemma}
\label{lem:singletons_to_singletons}
Assume that we have a Carath\'eodory distance-preserving map $\phi \colon D \to D'$
between two finite dimensional bounded symmetric domains,
such that $\rank D = \rank D'$.
Then, the induced map on the set of Busemann
points (see Section~\ref{sec:isometric_embeddings})
takes singletons to singletons.
\end{lemma}

\begin{proof}
As in the proof of Theorem~\ref{thm:rank_inequality}, the map $\phi$
induces an injective map between the tripotents of $D$ and those of $D'$
that preserves the order.

Let $e$ be a tripotent of $D$ of rank $s$.
Then, $e$ is contained in a maximal chain $e_1, \dots, e_r$, and
by Lemma~\ref{lem:the_rank_of_a_tripotent_is_its_index_in_every_maximal_chain},
$e = e_s$.
By Lemma~\ref{lem:sequence_of_tripotents},
the chain $\phi (e_1), \dots, \phi (e_r)$ is also maximal,
and hence the rank of $\phi (e)$ is $s$.
We have shown that $\phi$ preserves the rank of every tripotent.

The conclusion now follows on observing that the singleton parts are precisely
those corresponding to minimal tripotents, that is, tripotents of rank 1.
\end{proof}

\section{The Gromov product in a bounded symmetric domain}
Recall that singleton Busemann points correspond exactly to
the minimal tripotents, and play a central role in our analysis.
We introduce the following notation: given a minimal tripotent $e$,
we denote by $\Xi_e$ the associated singleton Busemann point.

We have the following expression, involving the Bergman operator
$B(\cdot,\cdot)$ defined in~(\ref{eqn:bergman_definition}),
for the Gromov product of two singleton Busemann points
of a bounded symmetric domain.
We take the origin $0$ to be the base point.

\begin{theorem}
\label{thm:formula_for_the_gromov_product_of_singletons}
Let $u$ and $v$ be minimal tripotents in a JB*-triple $V$.
Then,
\begin{equation*}
\gromprod {\Xi_u} {\Xi_v} {0}
    = \frac {1} {2} \log \frac {4} {\|P_2(u) B(u, v) v\|}.
\end{equation*}
\end{theorem}

\begin{proof}
The path $t\colon (0,1) \to D$ defined by $t \mapsto tu$ is a geodesic in $D$
and converges to $\Xi_u$ in the horofunction compactification.
The Carath\'eodory distance between $0$ and $tu$ is $d(0, tu) = \tanh^{-1}t$.
Likewise, the path $s\colon (0,1) \to D$ defined by $s \mapsto sv$ is a geodesic
and converges to $\Xi_v$. 
So, by Proposition~\ref{prop:singletons_of_bsd}, we have the following expression for the iterated limit
\begin{align*}
\lim_{t\to 1}\lim_{s\to 1}\gromprod {tu} {sv} {0}
    &= \lim_{t\to 1} \big(d(0, tu) - \Xi_v(tu)\big) \\
    &= \lim_{t\to 1} \frac {1} {2}
      \log \left(\Big( \frac {1+t} {1-t} \Big)
           \big\| B(tu, tu)^{-1/2} B(tu, v) v \big\|^{-1}\right).
\end{align*}
Using the fact that
\begin{equation*}
\lim_{t\to 1} (1 - t^2) B(tu, tu)^{-1/2} = P_2(u),
\end{equation*}
and that $B(tu, v)$ converges to $B(u, v)$
(see the proof of Theorem~4.4 of \cite{chu_cueto_avellaneda_lemmens}), we find that 
\[
\lim_{t\to 1}\lim_{s\to 1}\,\gromprod {tu} {sv} {0} = \frac {1} {2} \log \frac {4} {\|P_2(u) B(u, v) v\|}. 
\]

As the double limit $\lim_{t,s\to\infty}\, \gromprod {tu} {sv} {0}$ exists in $[0,\infty]$ by Proposition\,\ref{prop:alt_formula_for_gromov_prod},
it is equal to the iterated limit. 
\end{proof}

\begin{lemma}
\label{lem:pierce_one_is_empty}
Let $u$ be a minimal tripotent in a bounded symmetric domain,
and let $b \in V_1(u)$ be such that $\triprod {b} {b} {u} = 0$. Then, $b = 0$.
\end{lemma}
\begin{proof}
The assumption implies that $b$ is orthogonal to $u$.
But the orthogonality relation is symmetrical,
and so $b = 2\triprod {u} {u} {b} = 0$.
\end{proof}

In the next lemma, we obtain a useful formula for the Gromov product
$\gromprod {\Xi_u} {\Xi_v} {0}$ by splitting the minimal tripotent $v$
into its Peirce components with respect to $u$.
Here, we use $\real \mu$ to denote the real part of a complex number $\mu$,
and $\overline \mu$ to denote its conjugate.
Also, recall that if $u$ is a minimal tripotent and $b\in V_1(u)$,
then $\{b,b,u\} \in V_2(u)$ by the Peirce calculus,
and so $\{b,b,u\}=\lambda u$ for some $\lambda\in \C$,
since $u$ is minimal. 

\begin{lemma}
\label{lem:formula_for_gromov_product_in_terms_of_decomposition}
Let $u$ and $v$ be minimal tripotents in a bounded symmetric domain.
We decompose $v = a + b + c$ into its Peirce components
$a := P_2(u) v$, $b := P_1(u) v$, and $c := P_0(u) v$,
and let $\mu$ and $\lambda$ in $\C$ be such that $a = \mu u$
and $\triprod {b} {b} {u} = \lambda u$.
Then,
\begin{equation*}
\gromprod {\Xi_u} {\Xi_v} {0}
    = \frac {1} {2} \log \frac {2} {\big|\real \mu - |\mu|^2 - \lambda \big|}.
\end{equation*}
\end{lemma}

\begin{proof}
By the Peirce calculus, and the linearity and conjugate-linearity of the
triple product,
\begin{align*}
P_2(u) B(u, v) v
    &= P_2(u) \big(v - 2\{u, v, v\} + \{u, v, u\}\big) \\
    &= a - 2 \{u, a, a\} - 2 \{u, b, b\} + \{u, a, u\} \\
    &= \big( \mu - 2 |\mu|^2 - 2 \lambda + \overline\mu \big) u.
\end{align*}
The result now follows upon applying
Theorem~\ref{thm:formula_for_the_gromov_product_of_singletons}.
\end{proof}

It is easy to calculate the Gromov product when one of the minimal tripotents
is a complex multiple of the other.

\begin{lemma}
\label{lem:calculate_gromov_product_for_complex_multiples}
Let $u$ and $v$ be minimal tripotents in a bounded symmetric domain,
such that $v = \mu u$, for some $\mu \in\C$ with $|\mu| = 1$.
Then,
\begin{equation*}
\gromprod {\Xi_u} {\Xi_v} {0}
    = \frac {1} {2} \log \frac {2} {1 - \real \mu}.
\end{equation*}
\end{lemma}

\begin{proof}
This is a simple calculation using
Lemma~\ref{lem:formula_for_gromov_product_in_terms_of_decomposition}.
\end{proof}

Our strategy will be to relate algebraic properties of minimal tripotents
to the Gromov product of the associated Busemann points.
Since the Gromov product is preserved by distance-preserving maps,
the properties of the tripotents will be as well.

First, we characterise when two minimal tripotents are opposite one another.

\begin{proposition}
\label{prop:characterise_opposite}
Let $u$ and $v$ be minimal tripotents in a JB*-triple. Then, $v = -u$ if and only if $\gromprod {\Xi_u} {\Xi_v} {0} = 0$.
\end{proposition}
\begin{proof}
When $v = -u$, it is easy to see from
Lemma~\ref{lem:calculate_gromov_product_for_complex_multiples}
that $\gromprod {\Xi_u} {\Xi_v} {0} = 0$.

To prove the converse, assume that the latter equation is true.
So, by Theorem~\ref{thm:formula_for_the_gromov_product_of_singletons},
$\|P_2(u) B(u, v) v\| = 4$.
We have
\begin{equation*}
\label{eqn:three_terms}
P_2(u) B(u, v) v
    = P_2(u) \big(v - 2\{u, v, v\} + \{u, v, u\}\big).
\end{equation*}
Observe that 
the projections to the Peirce space $V_2(u)$ of each of the terms
$v$, $-\triprod {u} {v} {v}$, and $\triprod {u} {v} {u}$ is a complex multiple of $u$.
Moreover, each of these projections has norm at most $1$.
It follows that each of them must have norm exactly $1$,
and in fact must all be equal. Thus, we have
$P_2(u)v = -P_2(u) \triprod {u} {v} {v} = P_2(u) \triprod {u} {v} {u} = \mu u$,
for some $\mu \in \C$ with $|\mu|= 1$.

Write $a := P_2(u) v$, $b := P_1(u) v$, and $c := P_0(u) v$.
By the Peirce calculus,
\begin{equation}
\label{eqn:mu_equation}
\mu u = -P_2(u) \triprod {u} {v} {v}
= -\triprod {u} {a} {a} - \triprod {u} {b} {b}.
\end{equation}
Since $u$ is an eigenvector of both $a \smallsquare a$
and $b \smallsquare b$,
and these operators are Hermitian with non-negative spectrum,
we deduce that $\mu$ is negative, and hence equals $-1$.
So, $a = -u$.
From \eqref{eqn:mu_equation} again,
we get $\triprod {u} {b} {b} = 0$,
and hence by Lemma~\ref{lem:pierce_one_is_empty}, that $b = 0$.

Observe that, since $c$ is in $V_0(u)$, it is orthogonal to $u$.
Using the Peirce calculus again,
\begin{equation*}
a + c = v = \triprod {v} {v} {v}
    = \triprod {a} {a} {a} + \triprod {c} {c} {c}.
\end{equation*}
As $a = -u$, we have $\{a,a,a\} = a$, and so $c=\{c,c,c\}$. Hence, both $a$ and $c$ are tripotents.
Since $v$ is minimal, only one of them can
be non-zero, in the present case necessarily $a$.
We have shown that $v = -u$.
\end{proof}

Next we characterise orthogonality of minimal tripotents.

\begin{proposition}
\label{prop:characterise_orthogonality}
Let $u$ and $v$ be minimal tripotents in a JB*-triple. Then,
$u$ and $v$ are orthogonal if and only if
$\gromprod {\Xi_u} {\Xi_v} {0} = \gromprod {\Xi_u} {\Xi_{-v}} {0} = \infty$.
\end{proposition}

\begin{proof}
First assume that $u$ and $v$ are orthogonal.
In this case $B(u, v)$ is the identity map, and so we get 
$P_2(u) B(u, v) v = 0$. Using the formula for the Gromov product in
Theorem~\ref{thm:formula_for_the_gromov_product_of_singletons},
we see that $\gromprod {\Xi_u} {\Xi_v} {0} = \infty$.
The same conclusion also holds when $v$ is replaced by $-v$.

Now assume that $u$ and $v$ are minimal tripotents such that
the two equations in the statement of the proposition hold.
We use the same notation as in
Lemma~\ref{lem:formula_for_gromov_product_in_terms_of_decomposition}.
Since $\{u, b, b\}$ is in $V_2(u)$, it is equal
to $\lambda u$ for some $\lambda$ in $\C$.
So, $u$ is an eigenvector of $b \smallsquare b$. Since this operator
is Hermitian and has non-negative spectrum, $\lambda$ must be non-negative.
We also have that $a = \mu u$, for some $\mu\in\C$.
So, by Lemma~\ref{lem:formula_for_gromov_product_in_terms_of_decomposition},
the two equations can be expressed in the form
\begin{equation*}
\real \mu - |\mu|^2 - \lambda = 0
= -\real \mu - |\mu|^2 - \lambda.
\end{equation*}
Taking their difference, we get that $\real \mu$ is zero,
and hence that $|\mu|^2$ and $\lambda$ are also zero, since both
are non-negative.
So, we have that $a=0$ and, using Lemma~\ref{lem:pierce_one_is_empty},
that $b = 0$.
We have shown that $v = c$, and since $c$ is in $V_0(u)$,
it is orthogonal to $u$.
\end{proof}

Finally, we characterise when two minimal tripotents are related by a multiple
$\pm \mathrm{i}$. To do this, the following lemma will be useful.

\begin{lemma}
\label{lem:no_orthogonal_component}
Let $u$ and $v$ be minimal tripotents in a JB*-triple.
Assume there exists a frame with $u$ as one of its elements,
such that $v$ is orthogonal to each element of the frame apart from $u$.
Then, $P_0(u)v = 0$.
\end{lemma}

\begin{proof}
Let $e_1, \dots, e_r$ be the frame, with $e_1 := u$, and set $M:=\{0,1,\ldots,r\}$.
For each $i \in M$ with $i\neq 0$, we have the following chain of equivalences:
\begin{align*}
\text{$v$ is orthogonal to $e_i$}
    & \Longleftrightarrow P_1(e_i) v = P_2(e_i) v = 0 \\
    & \Longleftrightarrow \sum_{k \in M\setminus\{i\}} P_{ki}(e_1, \dots, e_r) v
          = P_{ii}(e_1, \dots, e_r) v = 0 \\
    & \Longleftrightarrow P_{ki}(e_1, \dots, e_r) v = 0,
\quad\text{for all $k \in M$}.
\end{align*}
Also, since $e_1, \dots, e_r$ is a frame,
$P_{00}(e_1, \dots, e_r) v = 0$.

Now observe that, as $v$ is orthogonal to each $e_i$ with $i\neq 1$, 
\begin{equation*}
P_0(u) v = P_{0}(e_1) v
    = \sum_{\substack{i,j\in M\setminus\{1\} \\ i \le j }}
          P_{ij}(e_1, \dots, e_r) v =0.
\qedhere
\end{equation*}
\end{proof}

\begin{lemma}
\label{lem:characterise_multiplication_by_i}
Let $u$ and $v$ be minimal tripotents in a JB*-triple.
Then, $v = \pm \mathrm{i} u$
if and only if both the following conditions hold:
\begin{itemize}
\item
there exists a frame with $u$ as one of its elements, such that $v$ is orthogonal to each
element of the frame apart from $u$;
\item
$\gromprod {\Xi_u} {\Xi_v} {0} + \gromprod {\Xi_u} {\Xi_{-v}} {0}
    = \log 2$.
\end{itemize}
\end{lemma}

\begin{proof}
Assume that $v = \pm \mathrm{i}u$. Take any frame $u_1,\dots,u_r$ with $u = u_1$.
Since orthogonality of two elements is preserved when one of them is
multiplied by a complex scalar, $v$ is orthogonal to each of $u_2, \dots, u_r$.
This establishes the first of the two conditions. The second follows from
Lemma~\ref{lem:calculate_gromov_product_for_complex_multiples}.

Now assume that $u$ and $v$ satisfy the two conditions in the statement
of the lemma.
We decompose $v = a + b + c$,
where $a := P_2(u) v$, $b := P_1(u) v$, and $c := P_0(u) v$.
By Lemma~\ref{lem:no_orthogonal_component},
the first condition implies that $c=0$.

There are complex numbers $\mu$ and $\lambda$ such that
$a = \mu u$ and $\triprod {b} {b} {u} = \lambda u$,
since these are elements of $V_2(u)$ and $u$ is minimal.
Moreover, $\lambda$ must be non-negative since
$b \smallsquare b$ is Hermitian and has non-negative spectrum.
Since $v$ is a tripotent and $c =0$, we have by the Peirce calculus
\begin{equation*}
a =  P_2(u)\{a+b,a+b,a+b\} = \triprod {a} {a} {a} + 2 \triprod {b} {b} {a}
    = |\mu|^2 \mu u + 2\lambda \mu u
    = |\mu|^2 a + 2\lambda a.
\end{equation*}
Thus, $|\mu|^2 + 2\lambda = 1$. This implies that $|\mu|^2 + \lambda \le 1$,
with equality only when $|\mu|=1$ and $\lambda=0$. Moreover, $(\real \mu)^2\leq 1$.

It can be seen from
Lemma~\ref{lem:formula_for_gromov_product_in_terms_of_decomposition} that 
the second condition we have imposed on $u$ and $v$
is equivalent to
\begin{equation*}
\left| \real\mu - |\mu|^2 - \lambda \right|\cdot
\left| -\real\mu - |\mu|^2 - \lambda \right|
    = \left| (|\mu|^2 + \lambda )^2 - ( \real \mu)^2\right| =1.
\end{equation*}
For the expression inside the modulus to take the value $-1$, one would need that
$|\mu|^2 + \lambda = 0$ and $\real\mu=1$, which is clearly impossible.
On the other hand, for the value $1$ to be obtained,
we need that $|\mu|^2 + \lambda = 1$ and $\real\mu=0$.
Combining this with the fact that $|\mu |^2 +2\lambda =1$,
gives that $\lambda =0$ and  $\mu$ is either $\mathrm{i}$ or $-\mathrm{i}$. We conclude that $\{b,b,u\} =0$, which implies that $b=0$ by Lemma~\ref{lem:pierce_one_is_empty}. This completes the proof.
\end{proof}

A flat in a bounded symmetric domain $D$ is a maximal embedded Euclidean space,
when one takes the Bergman distance on $D$.
Every flat has the same dimension, namely the rank $r$ of $D$.
If one takes the Kobayashi/Carath\'eodory distance
instead of the Bergman distance, then each flat is isometric to $\R^r$
with the $\ell_\infty$-norm, $\|\cdot\|_\infty$. 
Moreover, given a frame $e_1,\dots,e_r$ in $D$, the set
\begin{equation*}
F :=
    \big\{ \lambda_1 e_1 + \dots + \lambda_r e_r
        \mid \text{$\lambda_i \in (-1, 1)$, for all $i$} \big\}
\end{equation*}
is a flat. In fact, every flat containing the origin is of this form,
and in this case the isometry is the restriction of the exponential map,
$z\in V\mapsto \tanh(z)$, to the linear span of the flat;
see~\cite[Lemma 4.3 and Corollary 4.8]{loos_bounded_symmetric_domains_and_jordan_pairs}.

For each $a\in D$, the \emph{M\"obius transformation} $g_a \colon D\to D$
given by, 
\begin{equation*}
g_a(x) := a + B(a,a)^{1/2} (\id + x \smallsquare a)^{-1}(x),
\qquad\text{for all $x\in D$},
\end{equation*}
is a biholomorphic map, and hence an isometry of the Kobayashi/Carath\'eodory
distance and of the Bergman distance.
The derivative of $g_a$ at any point $b\in D$ has the following expression
in terms of the Bergman operator:
$g'_a(b) = B(a, a)^{1/2} B(b, -a)^{-1}$;
see~\cite[equation (3.2)]{chu_Jordan_structures_in_geometry_and_analysis}.
We will see that if $a$ and $b$ are points lying in a common flat
that contains $0$,
then $g_a$ and $g_b$ commute, and both maps leave the flat invariant.
\begin{lemma}
\label{lem:bergman_same_flat}
Let $e_1,\dots, e_r$ be a frame, and let $a := a_1 e_1 + \dots + a_r e_r$
and $b := b_1 e_1 + \dots + b_r e_r$ be two elements of the associated flat,
with all coefficients $a_i$ and $b_i$ in $\R$.
Set $a_0 := b_0 := 0$.
Then,
\begin{equation*}
B(a, b)
    = \sum_{0\leq i\leq j\leq r}
           \big( 1 - a_i b_i \big)
           \big( 1 - a_j b_j \big)
           P_{ij}(e_1,\dots, e_r).
\end{equation*}
\end{lemma}

\begin{proof}
Let $z \in P_{ij}$, with $i, j \in \{0, \dots, r\}$.
We use the convention that $e_0 := 0$.
Using the orthogonality of the $e_k$ and that $z$ is
an eigenvector of each $e_k \smallsquare e_k$, we get
\begin{equation*}
a \smallsquare b(z)
    = \sum_k a_k b_k \{e_k, e_k, z\}
    = \sum_k a_k b_k (\delta_{ik} + \delta_{jk}) z
    = \frac {1} {2} \big(a_i b_i + a_j b_j\big) z.
\end{equation*}
From the Jordan identity, we have, for $k, l, m, n \in \{1, \dots, r\}$,
\begin{align*}
\big\{e_k, \{e_m, z, e_n\}, e_l\big\}
    &= \big\{ \{z, e_m, e_k \}, e_n, e_l\big\} \\
    & \qquad\qquad
     + \big\{ e_k, e_n, \{ z, e_m, e_l \}\big\} 
          -\big\{ z, e_m, \{ e_k, e_n, e_l \}\big\}.
\end{align*}
So,
\begin{align*}
Q_a Q_b(z)
    &= \sum_{m, n, k, l} a_k a_l b_m b_n
          \big\{e_k, \{e_m, z, e_n\}, e_l \big\} \\
    &= \frac {1} {4} \sum_{m, n, k, l} a_k a_l b_m b_n
       \Big[
      \delta_{mk}\delta_{nl}(\delta_{im}+\delta_{jm})(\delta_{in}+\delta_{jn})
      \\
    & \qquad\qquad
    + \delta_{ml}\delta_{nk}(\delta_{im}+\delta_{jm})(\delta_{in}+\delta_{jn})
    - 2 \delta_{kn}\delta_{nl}\delta_{lm}(\delta_{im}+\delta_{jm})
       \Big] z \\
    &= \frac {1} {4}
        \big[ 2(a_i b_i + a_j b_j)^2 - 2(a_i^2 b_i^2 + a_j^2 b_j^2 ) \big] z \\
    &= a_i a_j b_i b_j z.
\end{align*}
Combining this with the result at the start, we have
\begin{equation*}
B(a, b) z
    = \big( 1 - a_i b_i - a_j b_j
          + a_i a_j b_i b_j \big) z
    = \big( 1 - a_i b_i \big) \big( 1 - a_j b_j \big) z,
\end{equation*}
and the conclusion follows.
\end{proof}

\begin{lemma}
\label{lem:add_mobius_maps}
Let $e_1,\dots, e_r$ be a frame,
and let
\begin{equation*}
a := \sum_i \tanh(\alpha_i) e_i,
\qquad
b := \sum_i \tanh(\beta_i) e_i,
\qquad\text{and}\quad
c = \sum_i \tanh(\alpha_i + \beta_i) e_i
\end{equation*}
be elements of the associated flat, with each $\alpha_i$ and $\beta_i$ in $\R$.
Then, $g_a \after g_b = g_c$.
\end{lemma}

\begin{proof}
For each $i$, write $a_i := \tanh \alpha_i$ and $b_i := \tanh \beta_i$.
Let $y = \sum_i y_i e_i$, with each $y_i := b_i/(1 + b_i a_i)$.
Observe that $(\id + b \smallsquare a) y = b$.
We conclude that $(\id + b \smallsquare a)^{-1} b = y$.

For any $z := \sum_i z_i e_i$, with each $z_i \in \R$,
we have
\begin{equation*}
B(a, a) z = \sum_i (1 - a_i^2)^2 z_i e_i.
\end{equation*}
It follows that
\begin{equation*}
B(a, a)^{1/2} z = \sum_i (1 - a_i^2) z_i e_i.
\end{equation*}
Applying this to $y$, we get
\begin{align*}
g_a(b)
    &= a + B(a, a)^{1/2} (\id + b \smallsquare a)^{-1} b \\
    &= \sum_i \Big( a_i + \frac {(1 - a_i^2) b_i} {1 + b_i a_i} \Big) e_i \\
    &= \sum_i \frac {a_i + b_i} {1 + b_i a_i} e_i \\
    &= \sum_i \tanh(\alpha_i + \beta_i) e_i \\
    &= c.
\end{align*}
So, the maps $g_a \after g_b$ and $g_c$ agree at $0$.

Let $z$ be in the joint Peirce space $V_{ij}$, with $i, j \in \{0, \dots, r\}$.
The derivative of $g_c$ at $0$ applied to $z$ is
\[ g'_c(0) z = B(c, c)^{1/2} z 
   = \left(\Big(1 - \Big(\frac {a_i + b_i} {1 + a_i b_i} \Big)^2\Big)
       \Big(1 - \Big(\frac {a_j + b_j} {1 + a_j b_j} \Big)^2\Big)\right)^{1/2} z,
\]
using Lemma~\ref{lem:bergman_same_flat}.
Similarly, the derivative of $g_a \after g_b$ at $0$ applied to $z$ is
\begin{align*}
(g_a \after g_b)'(0) z    &= g'_a(b) g'_b(0) z \\
    &= B(a, a)^{1/2} B(b, -a)^{-1} B(b, b)^{1/2} z \\
    &= \left( \frac {(1 - a_i^2) (1 - a_j^2) (1 - b_i^2) (1 - b_j^2)}
       {(1 + a_i b_i)^{2} (1 + a_j b_j)^{2}} \right)^{1/2} z.
\end{align*}
Elementary algebra shows that these two expressions are equal.
Using that $V$ decomposes as the sum of the Peirce spaces,
we conclude that $g_a \after g_b$ and $g_c$ have the same derivative at $0$.

It now follows by Cartan's uniqueness theorem that $g_a \after g_b = g_c$.
\end{proof}

We have seen that every frame gives rise to a flat.
It also defines a collection of minimal tripotents and their
associated Busemann points. We need to study how these objects are related.
Indeed, we will characterise when a point is in the flat
in terms of values of the Busemann points there.

According to~\cite[Lemma 1.6]
{friedman_russo_structure_of_the_predual_of_a_jbwstar_triple},
if $e$ is a tripotent in a JB*-triple $V$, and $x$ is an element with $\|x\|=1$
and $P_2(e) x = e$, then $P_1(e) x = 0$.
We will also need that if $x$ and $y$ in $V$ are orthogonal, then
\begin{equation}
\label{eqn:norm_is_max_if_vectors_are_orthogonal}
\|x + y\| = \max \big( \|x\|, \|y\| \big);
\end{equation}
see~\cite[Corollary 3.1.21]{chu_Jordan_structures_in_geometry_and_analysis}.

\begin{proposition}
\label{prop:characterise_both_horofunctions_equal_zero}
Let $e$ be a minimal tripotent, and let $x\in D$.
Then, $\Xi_{e}(x) = \Xi_{-e}(x) = 0$
if and only if $x \in V_0(e)$.
\end{proposition}

\begin{proof}
First, assume that $x\in D$ is in $V_0(e)$.
Using the Peirce calculus, we get that
\begin{equation*}
B(x, e)e
    = e - 2 \triprod {x} {e} {e} + \triprod {x} {\triprod {e} {e} {e}} {x}
    = e.
\end{equation*}
Similarly, $B(x, x) e = e$, which implies that $B(x, x)^{-1} e = e$.
Since $B(x, x)^{-1}$ is a positive Hermitian operator
(see~\cite[Lemma 1.2.22]{chu_Jordan_structures_in_geometry_and_analysis}),
it has positive spectrum,
and we deduce that $B(x, x)^{-1/2} e = e$.
Therefore,
\begin{equation*}
\Xi_e(x)
    = \frac {1} {2} \log \big\| B(x,x)^{-1/2} B(x, e) e \big\|
    = 0.
\end{equation*}
That $\Xi_{-e}(x)$ is also zero is proved similarly.

Now let $x\in D$ be such that $\Xi_{e}(x) = \Xi_{-e}(x) = 0$.
So, in particular,
\begin{equation*}
\big\| B(x,x)^{-1/2} B(x, e) e \big\| = 1.
\end{equation*}
Recall that $\|B(x, x)^{1/2}\| \le 1$;
see the discussion before Proposition 3.2.13
of~\cite{chu_Jordan_structures_in_geometry_and_analysis}.
We deduce that $\|B(x, e) e\| \le 1$.

Write $x = a + b + c$, with $a \in V_2(e)$, $b\in V_1(e)$, and $c\in V_0(e)$.
Since $e$ is minimal, we have $a = \mu e$, for some $\mu \in \C$.
By the Peirce calculus,
\begin{align*}
B(x, e) e
    &= e - 2 \mu e - 2 \triprod {b} {e} {e}
       + \mu^2 e + 2 \mu \triprod {b} {e} {e} + \triprod {b} {e} {b} \\
    &= (1 - \mu)^2 e
      - 2(1 - \mu) \triprod {b} {e} {e}
      + \triprod {b} {e} {b}.
\end{align*}
So, the projection of $B(x, e) e$ onto the Peirce 2-space $V_2(e)$ of $e$ is
$(1 - \mu)^2 e$. Since this projection does not increase the norm,
we have $|1 - \mu| \le 1$.

Using similar reasoning, we also get that $\|B(x, -e) e\| \le 1$, with
\begin{equation*}
B(x, -e) e
    = (1 + \mu)^2 e
      + 2(1 + \mu) \triprod {b} {e} {e}
      + \triprod {b} {e} {b},
\end{equation*}
and so $|1 + \mu| \le 1$.
We conclude that $\mu$ is zero, and hence so also is $a$.
So,
\begin{equation*}
B(x, e) e
    = e - 2 \triprod {b} {e} {e} + \triprod {b} {e} {b}.
\end{equation*}
The projection of this vector onto $V_2(e)$ is $e$, and hence its norm
is at least $1$. Combining this with what we had before,
its norm is actually equal to $1$. Applying~\cite[Lemma 1.6]
{friedman_russo_structure_of_the_predual_of_a_jbwstar_triple},
we get that $0 = P_1(e) B(x, e) e = -2\triprod {b} {e} {e} = -b$.
We have shown that $x \in V_0(e)$.
\end{proof}

Recall that the inverse hyperbolic tangent function is given by
\begin{equation*}
\tanh^{-1}x = \frac {1} {2} \log \frac {1 + x} {1 - x}.
\end{equation*}

\begin{lemma}
\label{lem:calculate_singleton_in_flat}
Let $e_1,\dots,e_r$ be a frame of a bounded symmetric domain,
and let $x := \lambda_1 e_1 + \dots + \lambda_r e_r$ 
be in the associated flat, with each $\lambda_i$ in $(-1, 1)$.
Then $\Xi_{e_i}(x) = -\tanh^{-1} \lambda_i$, for all $i$.
\end{lemma}

\begin{proof}
Since the $e_k$ are mutually orthogonal,
from the definition of the Bergman operator we have
\begin{equation*}
B(x, e_i) e_i
    = e_i - 2 \lambda_i e_i + \lambda^2_i e_i = (1 - \lambda_i)^2 e_i,
\qquad\text{for all $i$}.
\end{equation*}
The joint Peirce projection $P_{jk}(e_1, \dots, e_r) e_i$
equals $e_i$ when $j = k = i$,
and equals zero otherwise.
So, using Lemma~\ref{lem:bergman_same_flat},
we get $B(x, x)^{-1/2} e_i = (1 - \lambda_i^2)^{-1} e_i$.
Combining these formulae with Proposition~\ref{prop:singletons_of_bsd},
we see that
\begin{equation*}
\Xi_{e_i}(x)
    = \frac {1} {2} \log \frac {(1 - \lambda_i)^2} {1 - \lambda_i^2}
    = \frac {1} {2} \log \frac {1 - \lambda_i} {1 + \lambda_i}
    = -\tanh^{-1} \lambda_i.
\qedhere
\end{equation*}
\end{proof}

\begin{lemma}
\label{lem:translations_acting_on_singletons}
Let $e_1,\dots,e_r$ be a frame of a bounded symmetric domain $D$,
and let $x := \lambda_1 e_1 + \dots + \lambda_r e_r$ 
be in the associated flat, with each $\lambda_i$ in $(-1, 1)$.
Then,
\begin{equation*}
\Xi_{e_i}\big(g_x(y)\big) = \Xi_{e_i}(y) - \tanh^{-1}\lambda_i,
\qquad\text{for all $y\in D$ and all $i$}.
\end{equation*}
\end{lemma}

\begin{proof}
The flat $F$ associated with the frame $e_1, \dots, e_r$ is isometric
to the normed space $\R^r$ with the supremum norm $||\cdot||_\infty$.
The isometry is the exponential map at $0$,
namely $\exp_0 \colon \R^r \to F$, given by
\begin{equation*}
\exp_0 \big( (p_1, \dots, p_r) \big) = \tanh(p_1) e_1 + \dots + \tanh(p_r) e_r.
\end{equation*}
For each $i$, let $\mu_i := \tanh^{-1}\lambda_i$,
and define $M := \max_i |\mu_i|$.
The map $\exp_0^{-1}\after g_x\after\exp_0$  is a translation by $(\mu_1, \dots, \mu_r)$.

For simplicity, we consider the case where $i = 1$;
the other cases are similar.
The sequence $z_n : = \exp_0 ( (2n M, 0, \dots, 0))$ is an almost-geodesic
converging to $\Xi_{e_1}$ in the horofunction boundary.
Its image $y_n := g_{x}(z_n)$
under the map $g_{x}$ is also an almost-geodesic.

Let $w_n$ be the sequence obtained by taking alternate terms
of the sequences $z_n$ and $y_n$, that is,
$w_n := z_n$ for $n$ even, and $w_n := y_n$ for $n$ odd.
Let $m$ and $n$ be elements of $\N$ such that $m \le n$.
If $m$ and $n$ are either both even or both odd,
then $d(w_m, w_n) = 2(n - m) M$.
If $m$ is odd and $n$ is even, then $d(w_m, w_n) = 2(n - m) M + \mu_1$,
while
if $m$ is even and $n$ is odd, then $d(w_m, w_n) = 2(n - m) M - \mu_1$.
So, we see that $w_n$ is also an almost-geodesic.

It follows that the three sequences
converge to the same horofunction.
Therefore, using that $g_{x}$ preserves the distance,
\begin{align*}
\Xi_{e_1}\big(g_{x}(y)\big)
    &= \lim_{n\to\infty}
           \Big(d\big(g_{x}(y), y_n\big) - d(0, y_n)\Big) \\
    &= \lim_{n\to\infty} \big(d(y, z_n) - d(0, z_n)\big)
       - \lim_{n\to\infty} \big(d(g_{-x}(0), z_n) - d(0, z_n)\big) \\
    &= \Xi_{e_1}(y) - \Xi_{e_1}(-x).
\end{align*}
The result now follows upon applying
Lemma~\ref{lem:calculate_singleton_in_flat}.
\end{proof}

\begin{proposition}
\label{prop:characterise_sum_equal_zero}
Let $e_1,\dots,e_r$ be a frame of a bounded symmetric domain $D$,
and let $x\in D$. Then $\Xi_{e_i}(x) + \Xi_{-e_i}(x) = 0$, for all $i$,
if and only if $x$ is in the flat defined by $e_1,\dots,e_r$.
\end{proposition}

\begin{proof}
Let $x := \lambda_1 e_1 + \dots + \lambda_r e_r$
be in the flat, with each $\lambda_i$ in $(-1, 1)$,
and take $j \in\{1, \dots, r\}$.
By Lemma~\ref{lem:calculate_singleton_in_flat}, we have
$\Xi_{e_j}(x) = -\tanh^{-1}\lambda_j$.
But $-e_1,\dots,-e_r$ is also a frame, and it gives rise to the same flat.
With respect to this frame, the coordinates of $x$
are $(-\lambda_1,\dots,-\lambda_r)$.
Using the same lemma again, we get $\Xi_{-e_j}(x) = -\tanh^{-1}(-\lambda_j)$.
We now use that the inverse hyperbolic tangent is an odd function
to get that the sum of $\Xi_{e_j}(x)$ and $\Xi_{-e_j}(x)$ is zero.

To prove the converse, let $x\in D$ and assume that $\Xi_{e_j}(x)+\Xi_{-e_j}(x) =0$  for each $j$.
For each $j$, let $\mu_j := \Xi_{e_j}(x) = -\Xi_{-e_j}(x)$,
and define $\lambda_j := \tanh \mu_j$.
The maps $\{g_{\lambda_j e_j}\}$; $j$ commute, by Lemma  \ref{lem:add_mobius_maps}. 
From Lemma~\ref{lem:translations_acting_on_singletons}
we know that, for $z\in D$ and $j,k\in\{1, \dots, r\}$,
\begin{align*}
\Xi_{e_k} \big( g_{\lambda_j e_j}(z)\big) &= 
\begin{cases}
\Xi_{e_k}(z), \qquad & \text{if $j\neq k$}; \\
\Xi_{e_k}(z) - \mu_k, \qquad & \text{if $j = k$}
\end{cases} \\
\text{and}\qquad
\Xi_{-e_k} \big( g_{\lambda_j e_j}(z)\big) &= 
\begin{cases}
\Xi_{-e_k}(z), \qquad & \text{if $j\neq k$}; \\
\Xi_{-e_k}(z) + \mu_k, \qquad & \text{if $j = k$}.
\end{cases}
\end{align*}
The point
$
y := g_{\lambda_r e_r} \after \cdots \after g_{\lambda_1 e_1} (x)$ satisfies $\Xi_{e_k} (y) = \Xi_{-e_k} (y) = 0$, for all $k$.
From Proposition~\ref{prop:characterise_both_horofunctions_equal_zero},
we get that $y$ is in $V_0(e_k)$, for each $k$. Since the $e_k$ form
a frame, it follows that $y = 0$. But $0$ is in the flat defined by
the frame, and therefore so also is
$x = g_{-\lambda_1 e_1} \after \cdots \after g_{-\lambda_r e_r} (0)$.
\end{proof}

\section{Carath\'eodory distance-preserving maps}

In this section, we use the Gromov product to study Carath\'eodory
distance-preserving maps between bounded symmetric domains of equal rank.
We show that flats are mapped to flats and the Bergman distance is preserved. 

Recall that each Carath\'eodory distance-preserving map $\phi$ 
between bounded symmetric domains of equal rank takes singleton Busemann
points to other such points (see Lemma~\ref{lem:singletons_to_singletons})
and thus induces a map, which we have also denoted by $\phi$,
between their associated minimal tripotents.

\begin{lemma}
\label{lem:preserves_orthogonality}
Let $\phi\colon D \to D'$ be a Carath\'eodory distance-preserving map
between two finite-dimensional bounded symmetric domains of equal rank.
Assume that $\phi(0) = 0$.
Then, two tripotents $u$ and $v$ in $D$ are orthogonal if and only if
$\phi(u)$ and $\phi(v)$ are orthogonal in $D'$.
\end{lemma}

\begin{proof}
Proposition~\ref{prop:characterise_opposite} characterises when two minimal
tripotents are opposite one another in terms of the Gromov product
of their associated Busemann points. The latter is preserved by $\phi$,
and so $\phi(-u) = -\phi(u)$ and $\phi(-v) = -\phi(v)$.
The conclusion now follows from the characterisation of orthogonality in Proposition~\ref{prop:characterise_orthogonality},
again in terms of the Gromov product.
\end{proof}

\begin{lemma}
\label{lem:new_flats_mapped_to_flats}
Let $\phi\colon D \to D'$ be a Carath\'eodory distance-preserving map
between two finite-dimensional bounded symmetric domains of equal rank $r$.
Assume that $\phi(0) = 0$. Let $e_1,\dots,e_r$ be a frame in $D$,
and let $x = \lambda_1 e_1 + \dots + \lambda_r e_r$,
with each $\lambda_i \in (-1, 1)$. Then,
\begin{equation}
\label{eqn:coefficients_of_image_point}
\phi(x) = \lambda_1 \phi(e_1) + \dots + \lambda_r \phi(e_r).
\end{equation}
\end{lemma}

\begin{proof}
It follows from Lemma~\ref{lem:preserves_orthogonality} that
the $\phi(e_i)$ form a frame of $D'$.

As $\phi(0)=0$, we have for each $i$ that 
\begin{equation}\label{eq:opposite}
\Xi_{\phi(e_i)} \big(\phi(x) \big) = \Xi_{e_i}(x) = -\tanh^{-1} \lambda_i,
\end{equation}
by Lemma~\ref{lem:calculate_singleton_in_flat},
and a similar equation holds for the opposite tripotents $-e_i$.
In particular,
\begin{equation*}
\Xi_{\phi(e_i)} \big(\phi (x) \big) + \Xi_{-\phi(e_i)} \big( \phi (x) \big) = 0,
\qquad\text{for all $i$}.
\end{equation*}
So, by Proposition~\ref{prop:characterise_sum_equal_zero},
$\phi(x)$ lies in the flat defined by the $\phi(e_i)$'s.
This means that $\phi(x)$ can be expressed as a real linear combination
of the $\phi(e_i)$. Indeed, using \eqref{eq:opposite} and  Lemma~\ref{lem:calculate_singleton_in_flat} again, we see that \eqref{eqn:coefficients_of_image_point} must hold.
\end{proof}

This lemma has the following consequences.

\begin{lemma}
\label{lem:flats_mapped_to_flats}
Let $\phi\colon D \to D'$ be a Carath\'eodory distance-preserving map
between two finite-dimensional bounded symmetric domains of equal rank.
If $F$ is a flat in $D$, then $\phi(F)$ is a flat in $D'$.
\end{lemma}

\begin{proof}
Let $x\in F$.
The map $\psi := g_{-\phi(x)} \after \phi \after g_x$ preserves the
Carath\'eodory distance and satisfies $\psi(0) = 0$.
Moreover, $g_{-x}(F)$ is a flat containing $0$.
By Lemma~\ref{lem:new_flats_mapped_to_flats},
this flat is mapped by $\psi$ to another flat,
namely $g_{-\phi(x)}(\phi(F))$. We deduce that $\phi(F)$ is a flat.
\end{proof}

Recall that, if $e_1, \dots, e_r$ is a frame of $D$
and if $x\in D$ is given by $x = \lambda_1 e_1 + \dots + \lambda_r e_r$,
with each $\lambda_i \in (-1, 1)$,
then the Bergman distance between $0$ and $x$ is
\begin{equation*}
d_B(0, x)
    = \big(
      (\tanh^{-1}\lambda_1)^2 +  \dots + (\tanh^{-1}\lambda_r)^2
      \big)^{1/2}.
\end{equation*}

\begin{lemma}
\label{lem:distance_preserving_for_bergman_metric}
Let $\phi\colon D \to D'$ be a Carath\'eodory distance-preserving map
between two finite-dimensional bounded symmetric domains of equal rank.
Then, $\phi$ is also distance-preserving for the Bergman distance.
\end{lemma}

\begin{proof}
Let $x$ and $y$ be points in $D$.
The map $\psi := g_{-\phi(x)} \after \phi \after g_{x}$ preserves the
Carath\'eodory distance and satisfies $\psi(0) = 0$.
Let $e_1, \dots, e_r$ be a frame of $D$ such that
$z := g_{-x}(y) = \lambda_1 e_1 + \dots + \lambda_r e_r$,
with each $\lambda_i \in (-1, 1)$.
By Lemma~\ref{lem:new_flats_mapped_to_flats},
we have $d_B(0, \psi(z)) = d_B(0, z)$.
Using that $g_{\phi(x)}$ preserves the Bergman distance on $D'$
and $g_{x}$ preserves the Bergman distance on $D$,
we get that $d_B(\phi(x), \phi(y)) = d_B(x, y)$.
\end{proof}

In Theorem~11.1
of~\cite{helgason_differential_geometry_lie_groups_and_symmetric_spaces},
it is shown that a surjective distance-preserving map
from a Riemannian manifold onto itself is automatically a diffeomorphism
that preserves the Riemannian structure.
Actually, the same proof gives that a distance-preserving map
from one $C^\infty$ Riemannian manifold into another is $C^\infty$
and preserves the Riemannian structure, even if it is not surjective.

\begin{lemma}
\label{lem:c_infinity}
If $\phi\colon D \to D'$ be a Carath\'eodory distance-preserving map
between two finite-dimensional bounded symmetric domains of equal rank, then $\phi$ is a $C^\infty$ map.
\end{lemma}
\begin{proof}
By Lemma~\ref{lem:distance_preserving_for_bergman_metric},
the map $\phi$ preserves the Bergman distance.
Under this distance, the bounded symmetric domains are
$C^\infty$ Riemannian manifolds, and distance-preserving maps between
such manifolds are $C^\infty$.
\end{proof}

\section{Proof of the main results}

To prove Theorem~\ref{thm:respects_product} we need to capture
the structure of a bounded symmetric domain $D$ as a product irreducible factors,
in terms of its minimal tripotents.
To do this we define the following equivalence relation on the set of minimal tripotents
of a JB*-triple.
We say that $u \chainlinked v$ if there is a finite sequence $e_1, \dots, e_n$
of minimal tripotents such that $e_1 = u$ and $e_n = v$,
and such that no two consecutive elements are orthogonal.
Let $M_i$, $i\in I$, be the equivalence classes of minimal tripotents,
and let $V_i$ be the real linear span of $M_i$, for each $i$.
Observe that, if $i$ and $j$ are distinct, then every element of $V_i$ is orthogonal to every element of $V_j$,
and hence the number of distinct $V_i$'s is at most the rank of $V$. 

Recall that the set of minimal tripotents of a product of JB*-triples
is the union of the sets of minimal tripotents of the factors.
That is, each minimal tripotent is of the form
$(e_1, 0)$ or $(0, e_2)$, where $e_1$ and $e_2$ are minimal tripotents
of the respective factors.
Minimal tripotents coming from different factors are of course orthogonal.

A \emph{subtriple} of a JB*-triple is a closed subspace that is also closed
with respect to the triple product.

\begin{lemma}
\label{lem:in_same_component_if_non_ortho_chain}
Let $V$ be a finite-dimensional JB*-triple,
and let $V_i$ be the subspaces defined above.
Then, each $V_i$ is an irreducible  subtriple of $V$,
and $V = V_1 \oplus \dots \oplus V_n$.
\end{lemma}

\begin{proof}
If $u$ is a minimal tripotent and $\lambda\in \C$, with $|\lambda|=1$,
then $u\chainlinked \lambda u$.
It follows that each $V_i$ is a complex linear subspace of $V$.

Since every element of $V$ can be written as a linear combination
of minimal tripotents, it can also be written as a sum of elements,
one from each of the $V_i$.

To see that each $V_i$ is closed under triple products,
fix $i$ and let $u, v, w \in V_i$.
We can write $\{u, v, w\}$ as a real linear combination of
elements of the form $\{e_1, e_2, e_3\}$,
where each of the $e_j$ are minimal tripotents in $M_i$.
Consider such an element, and let $c$ be in $M_j$, with $j \neq i$.
So, $c$ is orthogonal to each of $e_1$, $e_2$, and $e_3$.
Therefore, by the Jordan identity,
\begin{equation*}
\big\{c, c, \{e_1, e_2, e_3\}\big\}
    = \big\{\{c,c,e_1\}, e_2, e_3\big\} - \big\{e_1, \{c,c,e_2\}, e_3\big\}
         + \big\{e_1, e_2, \{c,c,e_3\}\big\}
    = 0.
\end{equation*}
So, $c$ is orthogonal to $\{e_1, e_2, e_3\}$.
We deduce that $\{u, v, w\}$ is orthogonal to every minimal tripotent
apart from those in $M_i$, and it follows that $\{u, v, w\}$ lies in $V_i$.

If $V_i$ were reducible, for some $i$, then we could partition its minimal
tripotents into two subsets in such a way that every element of one subset was
orthogonal to every element of the other; however this is clearly impossible.
\end{proof}

Recall that if $D = D_1 \times \dots \times D_n$
is a product of bounded symmetric domains, and each $D_i$ has rank $r_i$,
then $D$ has rank $r_1 + \dots + r_n$, and every frame of $D$
has exactly $r_i$ minimal tripotents coming from $D_i$.
Let $x := (x_1, \dots, x_n)$ be in $D$, with each $x_i$ in $D_i$.
The M\"obius transformation $g_x$ decomposes as
$g_x (y)  = (g_{x_1} (y_1), \dots, g_{x_n} (y_n))$,
for all $y = (y_1, \dots, y_n)$ in $D$.

\begin{proof}[Proof of Theorem~\ref{thm:respects_product}]
We first establish the result for the map $\psi := g_{-\phi(0)} \after \phi$.
The statement for $\phi$ then follows immediately,
because M\"obius transformations act on each component separately.
Observe that $\psi(0) = 0$, so $\psi$ preserves the Gromov product.

Let $u$ and $v$ be minimal tripotents belonging to the same factor $D_i$
of $D$.
By Lemma~\ref{lem:in_same_component_if_non_ortho_chain},
we have $u \chainlinked v$, so there exists a sequence $u=e_1,\ldots,e_n=v$ of minimal tripotents such that no two consecutive elements are orthogonal.
By Lemma~\ref{lem:preserves_orthogonality},
no two consecutive elements of the sequence $\psi(e_1),\ldots,\psi(e_n)$
are orthogonal.
Hence, $\psi(u) \chainlinked \psi(v)$,
and it follows from Lemma~\ref{lem:in_same_component_if_non_ortho_chain} that $\psi(u)$ and $\psi(v)$ are in the same factor of $D'$. 
Define the map $\indexfn$ as follows. For each factor $D_i$ of $D$,
choose a minimal tripotent $e$ in $D_i$, and set $\indexfn(i) = k$,
where $D'_k$ is the factor of $D'$ containing $\psi(e)$.
From what we have seen in the previous paragraph, this map is well-defined.

To show that $\indexfn$ is surjective,
let $D'_k$ be a factor of $D'$ and take any frame $e_1, \dots, e_r$ of $D$.
So, $\psi(e_1),\dots,\psi(e_r)$ is a frame of $D'$, and hence
contains a minimal tripotent of $D'_k$, say $\psi(e_j)$.
We then have $\indexfn(l) = k$,
where $D_l$ is the factor of $D$ containing $e_j$.

Fix a factor $D'_i$ of $D'$, and denote by $P_i$ the projection onto
this factor.
For $x \in D_{\indexfn^{-1}(i)}$,
define $\psi_i(x) := P_i \psi(y)$,
where $y$ is any element of $D$ such that $y_{\indexfn^{-1}(i)} = x$.
To show that this is well-defined, take any such $y$,
and let $e_1, \dots, e_r$ be a frame such that
$y = \lambda_1 e_1 + \dots + \lambda_r e_r$, with each $\lambda_i \in (-1, 1)$.
We order the elements of the frame so that
$e_1, \dots, e_s$ are the ones contained in $D_{\indexfn^{-1}(i)}$,
with $s \le r$. By the definition of the map $J$,
the minimal tripotents $e_1, \dots, e_s$ are exactly the elements of the frame
that are mapped to tripotents of $D_{\indexfn^{-1}(i)}$.
Applying Lemma~\ref{lem:new_flats_mapped_to_flats},
we have that $P_i \psi(y) = \lambda_1 \psi(e_1) + \dots + \lambda_s \psi(e_s)$.
So we see that $\psi_i(y)$ only depends on the components
in $\indexfn^{-1}(i)$.
\end{proof}

\begin{lemma}
\label{lem:real_linear}
Let $\phi\colon D \to D'$ be a Carath\'eodory distance-preserving map
between two finite-dimensional bounded symmetric domains of equal rank $r$,
contained in JB*-triples $V$ and $V'$, respectively.
If $\phi(0) = 0$, then $\phi$ is the restriction to $D$ of a real linear map
from $V$ to $V'$.
\end{lemma}

\begin{proof}
By Lemma~\ref{lem:c_infinity}, the map $\phi$ is $C^\infty$.
Denote by $D\phi(0) \colon V \to V'$ its derivative at $0$,
which is a real linear map.
Let $x\in D$. So, there exists a frame $e_1, \dots, e_r$
such that $x = \lambda_1 e_1 + \dots + \lambda_r e_r$,
with each $\lambda\in(-1, 1)$.
For each $t\in [-1,1]$, let
\begin{equation*}
\gamma(t) := \sum_i t \lambda_i e_i.
\end{equation*}
So, $\gamma$ is a smooth curve such that $\gamma(0) = 0$ and $\gamma(1) = x$.
Its tangent vector at $0$ is $x$.
By Lemma~\ref{lem:new_flats_mapped_to_flats},
\begin{equation*}
\phi\big(\gamma(t)\big) := \sum_i t \lambda_i \phi(e_i),
\qquad\text{for all $t\in (-1,1)$}.
\end{equation*}
The tangent vector of this curve at $0$ is $\phi(x)$.
Therefore, $D\phi(0)$ and $\phi$ agree at $x$.
We conclude that $\phi$ is the restriction of $D\phi(0)$ to $D$.
\end{proof}

\begin{lemma}
\label{lem:singletons_map_is_continuous}
Let $\phi \colon D \to D'$ be a Carath\'eodory distance-preserving map
between two finite dimensional bounded symmetric domains of equal rank.
Assume that $\phi(0) = 0$.
Then, the induced map on the set of minimal tripotents is continuous.
\end{lemma}

\begin{proof}
Let $e$ be a minimal tripotent of the JB*-triple $V$ of which $D$ is the
open unit ball, and let $\lambda \in (-1, 1)$.
By considering any frame containing $e$
and using Lemma~\ref{lem:new_flats_mapped_to_flats},
we get that $\phi(\lambda e) = \lambda \phi(e)$.

Now fix $\lambda \in (0, 1)$,
and let $e_n$ be a sequence of minimal tripotents of $V$ converging to $e$
in the norm topology of $V$. So, $\lambda e_n$ converges to $\lambda e$.
Hence, $\phi(\lambda e_n)$ converges to $\phi(\lambda e)$,
since $\phi$ is continuous on $D$. Therefore, by what we have seen above,
$\lambda \phi(e_n)$ converges to $\lambda\phi(e)$, and the conclusion follows
upon dividing by $\lambda$.
\end{proof}

\begin{proof}[Proof of Theorem~\ref{thm:holomorphic_or_antiholomorphic}]
By composing $\phi$ with a M\"obius transformation if necessary,
we can assume that $\phi(0) = 0$.

Recall that $\phi$ induces a map, also denoted by $\phi$,
from the set of Busemann points of $D$ to those of $D'$
with the property that every almost geodesic converging
to a Busemann point $\xi$ of $D$ is mapped to an almost-geodesic
converging to $\phi(\xi)$;
see Section~\ref{sec:isometric_embeddings}.
Moreover, by Lemma~\ref{lem:singletons_to_singletons},
singletons are mapped to singletons.
Since the singletons are in one-to-one correspondence with the minimal
tripotents, we get an induced map from the minimal tripotents of $D$
to those of $D'$.
By Lemma~\ref{lem:singletons_map_is_continuous}, this map,
still denoted by $\phi$, is continuous.

Combining~Lemma~\ref{lem:characterise_multiplication_by_i} with
Propositions~\ref{prop:characterise_opposite}
and~\ref{prop:characterise_orthogonality},
and using that $\phi$ preserves the Gromov product, we see that if
$u$ and $v$ are minimal tripotents of $D$ satisfying $u = \pm \mathrm{i} v$,
then $\phi(u) = \pm \mathrm{i} \phi(v)$.
In other words, for each minimal tripotent $e$ of $D$, we have that either
$\phi(\mathrm{i}e) = \mathrm{i}\phi(e)$ or $\phi(\mathrm{i}e) = -\mathrm{i}\phi(e)$.
Define $c(e)$ to be $\mathrm{i}$ in the former case, and $-\mathrm{i}$ in the latter.
Thus, we obtain a function $c$ from the set $E$ of minimal tripotents of $D$
to $\{\mathrm{i}, -\mathrm{i}\}$,
This map is continuous since $\phi$ is continuous on $E$.

We have assumed that $D$ is irreducible, and therefore $E$ is connected.
We deduce that  the map $c$ is constant, that is,
takes either only the value $\mathrm{i}$ or only the value $-\mathrm{i}$
on the whole of $E$.
We will treat the former case and show that $\phi$ is holomorphic.
In the latter case, it can be shown in the same way that $\phi$ is
anti-holomorphic.

Let $x\in D$. So, there exists a frame $e_1, \dots, e_r$
such that $x = \lambda_1 e_1 + \dots + \lambda_r e_r$,
with each $\lambda_i\in(-1, 1)$.
Hence,
$\mathrm{i}x = \lambda_1 \mathrm{i} e_1 + \dots + \lambda_r \mathrm{i}e_r$.
Applying Lemma~\ref{lem:new_flats_mapped_to_flats}, and using that $c$
is identically equal to $\mathrm{i}$, we get
\begin{align*}
\phi(\mathrm{i}x)
  &= \lambda_1 \phi(\mathrm{i} e_1) + \dots + \lambda_r \phi(\mathrm{i} e_r) \\
  &= \lambda_1 \mathrm{i} \phi(e_1) + \dots + \lambda_r \mathrm{i} \phi(e_r) \\
  &= \mathrm{i} \phi(x).
\end{align*}

From Lemma~\ref{lem:real_linear}, we know that $\phi$ on $D$ is the restriction
of a real linear map.
We have proved that it is actually the restriction of a complex
linear map, and hence it is holomorphic.
\end{proof}

Kaup~\cite{kaup_a_riemann_mapping_theorem_for_bounded_symmetric_domains_in_complex_banach_spaces}
showed that every \emph{surjective} complex linear map between JB*-triples
that preserves the norm is a triple homomorphism.
This is not necessarily true however for maps that are not surjective.
Nevertheless, Chu and Mackey~\cite{chu_mackey_isometries_between_jbstar_triples}
have shown the following.

\begin{theorem}[Chu---Mackey]
\label{thm:chu_mackey}
Let $\phi \colon V \to V'$ be a complex linear map between JB*-triples
that preserves the norm. Denote by $D$ and $D'$ the open unit balls
of $V$ and $V'$, respectively. Then, $\phi$ is a triple homomorphism
if and only if $\phi(D)$ is invariant under the M\"obius transformation
$g_{\phi(x)}$, for all $x\in D$.
\end{theorem}

For each $x\in D$, denote by $S_x$ the geodesic symmetry in $x$.
This means that, for each $y\in D$, the points $y$, $x$, and $S_x (y)$
lie on a Bergman geodesic, and $d_B(y, x) = d_B(x, S_x (y))$,
where $d_B$ is the Bergman distance.

\begin{lemma}
\label{lem:relation_mobius_flip}
If $w$ and $z$ in $D$ are such that $z = S_w (0)$,
then $g_z = S_w \after  S_0$.
\end{lemma}

\begin{proof}
Let $w = w_1 e_1 + \dots + w_r e_r$ be written in terms of
some frame $e_1, \dots, e_r$, with each $w_i \in (-1, 1)$.
So, $z = z_1 e_1 + \dots + z_r e_r$, where $z_i = 2 w_i / (1 + w_i^2)$
for all $i$.
By Lemma~\ref{lem:add_mobius_maps}, $g_z = g_w \after g_w$.
We deduce that $g_z$ maps $-w$ to $w$.
Observe that the same is also true for $S_w \after  S_0$.

Now we compare the derivatives at $-w$.
Let $x$ be in the joint Peirce space $V_{ij}$,
for some $i,j\in\{0, \dots, r\}$.
From Lemma~\ref{lem:bergman_same_flat},
\begin{equation*}
g'_z(-w) x = B(z, z)^{1/2} B(-w, -z)^{-1} x 
 = \frac {\big(1 - z_i^2\big)^{1/2} \big(1 - z_j^2\big)^{1/2}}
             {(1 - w_i z_i) (1 - w_j z_j)} x = x.
\end{equation*}
We deduce that $g'_z(-w)$ is the identity map.
Since $(S_w \after  S_0)'(-w)$ is also the identity map,
the two derivatives are equal at $-w$.
The result now follows from Cartan's uniqueness theorem.
\end{proof}

\begin{lemma}
\label{lem:preserves_reflection}
Let $\phi\colon D \to D'$ be a Bergman distance-preserving map
between two finite-dimensional bounded symmetric domains.
Then, $\phi(S_x (y)) = S_{\phi(x)} (\phi(y))$, for all $x$ and $y$ in $D$.
\end{lemma}

\begin{proof}
The points $y$, $x$, and $S_x (y)$ lie equally spaced along a Bergman geodesic
in $D$. Therefore their images $\phi(y)$, $\phi(x)$, and $\phi(S_x (y))$ lie
equally spaced along a Bergman geodesic in $D'$.
The conclusion follows.
\end{proof}

\begin{proof}[Proof of Theorem~\ref{thm:triple_homomorphism}]
By Lemma~\ref{lem:real_linear}, the map $\phi$ is the restriction of a real
linear map, which we also denote by $\phi$, between $V$ and $V'$.
So it agrees with its derivative at the origin, which is a complex linear map
since $\phi$ is assumed to be holomorphic.

Let $x = \lambda_1 e_1 + \dots + \lambda_r e_r$ be in $D$,
with $e_1,\dots,e_r$ a frame of $D$, and each $\lambda_i \in (-1, 1)$.
The norm of $x$ is $||x|| = \max(|\lambda_1|, \dots, |\lambda_r|)$,
and a similar expression holds for the norm in $V'$.
From Lemma~\ref{lem:new_flats_mapped_to_flats}, we get that the norm
is preserved by $\phi$ for elements of $D$.
The same is true for all elements of $V$, by linearity.

Let $z$ be in $D$. By Lemma~\ref{lem:preserves_reflection},
we have $S_{\phi(z)} (\phi(w)) = \phi(S_z (w))$, for all $w\in D$.
This shows that the set $\phi(D)$ is invariant under the point symmetry
$S_{\phi(z)}$, for any $z\in D$.

Let $x$ be in $D$, and let $y$ be the midpoint of $0$ and $x$ along the unique
Bergman geodesic between these two points. So, $S_y (0) = x$.
From Lemma~\ref{lem:preserves_reflection},
we get that $S_{\phi(y)}(0) = \phi(x)$.
By Lemma~\ref{lem:relation_mobius_flip},
this implies that $g_{\phi(x)} = S_{\phi(y)} \after S_0$.
Since it is a composition of maps that
each leave $\phi(D)$ invariant, $g_{\phi(x)}$ also leaves $\phi(D)$ invariant.

The conclusion now follows upon applying Theorem~\ref{thm:chu_mackey}.
\end{proof}

\noindent{\large\bf Declarations} 
\paragraph{\bf Conflict of interest statement} On behalf of all authors, the corresponding author states that there is no
Conflict of interest. 
\paragraph{\bf  Data availability statement}  The manuscript has no associated data. 
\bibliographystyle{plain}
\bibliography{hermitian}

\end{document}